\documentclass[a4paper, 11pt, onesided]{preprint}
\pdfoutput=1 %

\usepackage{tikz}
\usetikzlibrary{arrows, cd, backgrounds}

\usepackage[cachedir=biri-minted, frozencache=true]{minted}
\colorlet{CodeBackground}{black!5!white} %
\setminted{encoding=utf8, mathescape, labelposition=topline, frame=lines,
  fontsize=\footnotesize, xleftmargin=2em, xrightmargin=2em}
\setminted[macaulay2]{label={\TT{Macaulay2}}}
\setminted[mathematica]{label={\TT{Mathematica}}}

\usepackage{multibib}
\newcites{soft}{Mathematical software}

\usepackage{caption}
\usepackage{subcaption}
\usepackage{wrapfig}

\usepackage[en-US]{datetime2}
\usepackage{halloweenmath}
\DTMsavenow{now}
\newcommand{\hallow}[1]{%
  \ifthenelse{\equal{\DTMfetchmonth{now}}{10} \AND \(\equal{\DTMfetchday{now}}{31} \OR \equal{\DTMfetchday{now}}{30}\)}{%
    {\ensuremath{#1}}%
  }{}%
}

\makeatletter
\newenvironment{sidebyside}[1][\linewidth]{%
  \xdef\@save@parindent{\the\parindent}%
  \begin{minipage}{#1}%
}{%
  \end{minipage}%
}
\newenvironment{sidebox}[2][c]{%
  \begin{minipage}[#1]{#2}%
  \parindent=\@save@parindent\relax%
}{%
  \end{minipage}%
}
\makeatother

\NewDocumentCommand{\Arrow}{O{}o}{%
  \mathbin{%
    \tikz [baseline=-0.25ex, -latex]
      \draw [#1] (0pt,0.5ex) -- (1.3em,0.5ex)
        \IfNoValueF{#2}{node [midway, above, inner sep=0pt, outer sep=0pt] () {\scriptsize#2}}
      ;
  }%
}
\renewcommand{\to}{\mathrel{\Arrow[->]}}
\newcommand{\xto}[1]{\mathrel{\Arrow[->][#1]}}

\renewcommand{\into}{\mathrel{\Arrow[right hook->]}}

\newcommand{\iso}{\mathrel{\Arrow[<->]}}
\newcommand{\ratto}{\mathrel{\Arrow[->, densely dashed]}}

\let\vec\relax
\DeclareMathOperator{\vec}{vec}

\newcommand*\revcolon{%
  \nobreak
  \mskip6mu plus1mu
  \mathpunct{}%
  \nonscript
  \mkern-\thinmuskip
  {:}%
  \mskip2mu
  \relax
}

\usepackage{relsize}

\NewDocumentCommand{\ideal}{e{_}m}{\IfValueT{#1}{#1}\langle#2\rangle}
\NewDocumentCommand{\cone}{e{_}m}{\IfValueT{#1}{#1}[#2]}
\newcommand{\SOS}[1]{\mathlarger{\Sigma}{#1}^2}
\NewDocumentCommand{\monoid}{m}{\mathlarger{\Pi}(#1)}

\DeclareMathOperator{\Sper}{Sper}

\newcommand{\PD}{\RM{PD}}
\newcommand{\Sym}{\RM{Sym}}
\newcommand{\Stab}{\RM{Stab}}

\DeclareMathOperator{\Var}{Var}
\DeclareMathOperator{\Cov}{Cov}

\DeclareMathOperator{\pos}{pos}

\DeclareMathOperator{\num}{num}

\newcommand\<[1]{\ck{#1}}

\DeclareMathOperator{\pa}{pa}

\title[Real birational implicitization for statistical models]{%
Real birational implicitization\\%
for statistical models}

\author{Tobias Boege \and Liam Solus}

\address[T.B.~\hallow{\pumpkin}]{Department of Mathematics and Statistics, UiT The Arctic University, %
Tromsø, Norway}
\email{post@taboege.de}

\address[L.S.~\hallow{\mathbat}]{Department of Mathematics, KTH Royal Institute of Technology, Stockholm, Sweden}
\email{solus@kth.se}


\date{\today}

\begin{document}

\begin{abstract}
We derive an implicit description of the image of a semialgebraic set under a birational map, provided that the denominators of the map are positive on the set.
For statistical models which are globally rationally identifiable, this yields model-defining constraints which facilitate model membership testing, representation learning, and model equivalence tests.
Many examples illustrate the applicability of our results. The implicit equations recover well-known Markov properties of classical graphical models, as well as other well-studied equations such as the Verma constraint. They also provide Markov properties for generalizations of these frameworks, such as colored or interventional graphical models, staged trees, and the recently introduced Lyapunov models.
Under a further mild assumption, we show that our implicit equations generate the vanishing ideal of the model up to a saturation, generalizing previous results of Geiger, Meek and Sturmfels, Duarte and Görgen, Sullivant, and others.
\end{abstract}

\maketitle
\vspace{-\baselineskip}

\section{Introduction}
\label{sec: intro}

Statistical models are often defined on a relatively simple and intuitive semialgebraic parameter space $\Theta \subseteq \BB R^n$. A rational parametrization map $\alpha \colon \Theta \to \BB T \cong \BB R^m$ then takes each parameter vector $\theta$ to an alternative parametrization $\alpha(\theta) = t \in \BB T$ which is more amenable to statistical analysis; for example, $\BB T$ is often a space of sufficient statistics. %
We thus identify the statistical model with the image $\CC M = \alpha(\Theta) \subseteq \BB T$.
Subsequent statistical learning tasks, including hypothesis testing for model membership, model equivalence testing and representation learning, require an implicit description of~$\CC M$ via a collection of model-defining constraints; see, for example, \cite{koller2009probabilistic, pearl2009causality, murphy2012machine}.
Since $\CC M$ is the image of a semialgebraic set under a rational map, the Tarski--Seidenberg theorem guarantees that it is a semialgebraic set itself.  It is therefore described implicitly by polynomial equations and inequalities, providing the desired set of constraints.

The task of finding the polynomial constraints defining $\CC M$ from $\alpha$ and an implicit description of~$\Theta$ is more generally known as the \emph{implicitization problem} in algebraic geometry.
Classically, one assumes that the Zariski closure~of~$\Theta$ (and hence that of~$\CC M$) is irreducible and uses algorithms from elimination theory to derive generators of the prime ideal of polynomials vanishing on $\CC M$; see~\cite[Chapter~3]{CoxLittleOShea}. Polynomial inequalities, on the other hand, have only been worked out in special cases, for instance \cite{DrtonYu, ZwiernikSmith, AllmanRhodesTaylor, MontufarSeigal}.

In this paper we solve the implicitization problem under the additional assumption that $\CC M$ satisfies a version of global rational identifiability, namely that $\alpha$ viewed as a rational map $\BB R^n \ratto \BB T$ admits a rational inverse~$\beta\colon \BB T \ratto \BB R^n$.
We call these models \emph{ambirational}.
Geometrically, this means that $\Theta$ and $\CC M$ are birationally equivalent \emph{through} a birational isomorphism of their ambient spaces. %
Under this assumption and if $g$ vanishes on $\Theta$, then its pullback $\beta^*(g)$ is a vanishing rational function on~$\CC M$. It follows that its numerator vanishes and birationality guarantees that every vanishing polynomial on~$\CC M$ arises in this way. In other words, the vanishing ideal of~$\CC M$ is generated by the numerators of the generators of the vanishing ideal of $\Theta$ under $\beta^*$, up to saturation at the denominators of~$\beta^*$.
We can derive a similar characterization of the polynomial inequalities on $\CC M$ in terms of $\Theta$ and $\beta^*$ under the additional assumption --- which is natural in the statistical context --- that the denominators of $\beta^*$ are strictly positive on~$\CC M$.
Our~main result \Cref{thm:Biri} makes this idea rigorous and its supporting lemmas work in great generality.

The polynomial constraints defining $\Theta$ lead directly to constraints defining $\CC M$ via the numerators of~$\beta$. Due to a general desire for parameter interpretability in statistics, $\Theta$ is typically simple (oftentimes just a polyhedral cone), and thus $\CC M$ is ``simple by proxy''. We use this idea in \Cref{sec: stats} to generalize the notion of a Markov property in graphical modeling to all ambirational models. We observe that the resulting Markov property is amenable to recently-developed hypothesis tests for model membership \cite{sturma2024testing}, while also providing an algorithm for checking model equivalence. While the former has immediate applications in real data problems, the latter provides a useful tool in the study of open problems regarding representation learning for generalizations of graphical models.
We highlight that previous results regarding the saturation of ideals generated by Markov properties for Gaussian Bayesian networks \cite{ColoredDAG, Sullivant}, discrete Bayesian networks \cite{geiger2006toric} and staged trees \cite{duarte2020equations, duarte2020algebraic} are special cases of a general property of ambirational models (\Cref{thm:VanishingIdeal}), thereby unifying these results in a simple manner.

As ambirational models are ubiquitous in statistics, the results in this paper are generally applicable. In \Cref{sec:Ambirational}, we highlight popular examples including linear concentration models, Gaussian $\RM{MTP}_2$ models, generalizations of graphical and causal models, families of latent variable models and staged trees. In cases of well-studied models, the ambirational Markov property (see \Cref{subsec: markov properties}) typically recovers well-known graphical Markov properties. In cases where there is no known Markov property, \Cref{thm:Biri} gives a simple way to obtain one. This is exemplified in \Cref{ex:Lyapunov}, where we derive a Markov property for an instance of the recently introduced Lyapunov models \cite{IdentLyapunov}. We also observe how additional constraints can be easily added to ambirational models, yielding Markov properties for use in specific applications where these additional constraints are fitting to the available expert knowledge.

\section{Preliminaries}
\label{sec: prelims}

Familiarity with basic commutative and real algebra is assumed. Standard sources include \cite{AtiyahMacdonald,Kemper} and \cite{RealAlgebra,LamReal}. For an introduction to algebraic statistics \cite{Sullivant} is recommended. We recall some basic definitions and results in this section while setting~up~notation.

A ring is always commutative and has a multiplicative identity~$1$. If $A$ and $B$ are subsets of a ring, we use notation such as $A+B \defas \Set{a+b : a \in A, b \in B}$ or $A^2 \defas \Set{a^2 : a \in A}$ to extend operations to sets. Thus an ideal in a ring $A$ is a subset $\SR I \subseteq A$ with $\SR I + \SR I \subseteq \SR I$ and $A \cdot \SR I \subseteq \SR I$. For~any subset $B \subseteq A$ the smallest ideal of $A$ containing $B$ is denoted by $\ideal_A{B}$. It consists of all finite sums of the form $\sum_{i=1}^k a_i b_i$ with $a_i \in A$ and $b_i \in B$.

\subsection{Real algebra}

Let $A$ be a ring. A \emph{cone} in $A$ is a subset $\SR P \subseteq A$ with $\SR P + \SR P \subseteq \SR P$, $\SR P \cdot \SR P \subseteq \SR P$ and $A^2 \subseteq \SR P$. %
Just as ideals capture algebraic properties of ``vanishing'', cones are made to capture properties of ``non-negativity''. The smallest cone in any ring $A$ is the \emph{sums of squares cone} $\SOS{A}$ which consists of all finite sums $\sum_{i=1}^k a_i^2$ with $a_i \in A$.
Let $\SR P$ be any cone and $B \subseteq A$. The cone \emph{generated} by $B$ over $\SR P$ is denoted by $\cone_{\SR P}{B}$; it is the smallest cone containing $\SR P$ and $B$ and consists of all finite sums of the form $\sum_{i=1}^k p_i b_i$ where $p_i \in \SR P$ and each $b_i$ is a product of finitely many (distinct) elements of~$B$. The smallest cone containing $B$ is thus $\cone_{\SOS{A}}{B}$.

A cone $\SR P$ is \emph{proper} if $-1 \not\in \SR P$. If $\SR P$ is proper and $\SR P \cup -\SR P = A$, then $\SR P$ is also referred to as an \emph{ordering} because it establishes a total ordering $a \le b \iff b-a \in \SR P$ on $A$ which is compatible with addition and multiplication. In this case the \emph{support} $|\SR P| \defas \SR P \cap -\SR P$ is an ideal. Thinking of $\SR P$ as a set of non-negative elements, $|\SR P|$ is the set of elements considered zero and we use $\SR P^+ \defas \SR P \setminus -\SR P$ to denote the elements considered strictly positive.

An ideal $\SR I$ is \emph{real} if $\sum_{i=1}^k a_i^2 \in \SR I$ implies that each $a_i \in \SR I$. The real ideals in $\BB R[x_1, \dots, x_n]$ are precisely the vanishing ideals of subsets of $\BB R^n$. A cone $\SR P$ is \emph{prime} if $ab \in \SR P$ only if $a \in \SR P$ or $-b \in \SR P$. Prime cones are orderings. Moreover, the support of a prime cone $\SR P$ is a real prime ideal. Thus $\SR P$ even induces an ordering on the field of fractions of the quotient ring $A / |\SR P|$. The set of prime cones forms the \emph{real spectrum} $\Sper A$; more details can be found in \cite[Chapters~4~and~7]{RealAlgebra}.

\subsection{Localization and saturation}
\label{sec: localization}

A \emph{multiplicative set} $S$ in a ring $A$ is a submonoid of the multiplicative monoid~$A$ which does not contain~$0$. In particular, it contains~$1$ and is closed under multiplication. We shall also refer to a multiplicative set more concisely as a \emph{monoid} in~$A$. The monoid generated by a subset $B \subseteq A$ is denoted by $\monoid{B}$ and consists of all finite products of the form $\prod_{i=1}^k b_i^{d_i}$ with $b_i \in B$ and $d_i \in \BB N$.

Given a monoid $S \subseteq A$, the \emph{localization} $S^{-1} A \defas \Set{ \sfrac{a}{s} : a \in A, s \in S }$ is a ring with the usual equality, addition and multiplication of fractions; see~\cite[Chapter~3]{AtiyahMacdonald} for details. If~$S$ is free of zero divisors of~$A$, then the canonical homomorphism $A \to S^{-1} A$, $a \mapsto \sfrac{a}{1}$ is an embedding. In~this case, subsets of $A$ are often identified with their images under this embedding.
While our techniques can be generalized to this setting, below we will always make the simplifying assumption that $A$ is an integral domain. In applications, $A$ is often just a polynomial ring~over~a~field.

For a fraction $\sfrac{a}{s} \in S^{-1} A$, the numerator $\num(\sfrac{a}{s}) \defas a S$ is determined only up to multiplication with~$S$. Frequently, the chosen representative of the coset $aS$ will not matter and, by abuse of notation, we treat $\num(\sfrac{a}{s})$ as an element of~$A$.

An ideal or a cone in $A$ generates an ideal or cone, respectively, in $S^{-1} A$. This object in $S^{-1} A$ is known as the \emph{extension} of the object in~$A$. Conversely, given an ideal or cone in $S^{-1} A$, the intersection with $A$ is an ideal or cone, respectively, in~$A$, which is referred to as the \emph{contraction}. The contraction of the extension is usually larger than the starting object. It is described by saturation: for any $B \subseteq A$, the \emph{saturation of $B$ at $S$} is $B:S \defas \Set{a \in A : \exists s \in S : sa \in B}$. Saturation at a fixed monoid $S$ is an order-theoretic \emph{closure} operation on subsets of $A$, i.e., it is
\begin{description}[itemsep=-0.3em]
\item[extensive] $B \subseteq B:S$,
\item[monotone] $B \subseteq C$ implies $B:S \subseteq C:S$, and
\item[idempotent] $(B:S):S = B:S$.
\end{description}
The~set~$B$ is \emph{saturated at $S$} if $B:S = B$, i.e., $sa \in B$ with $s \in S$ necessitates $a \in B$.
The~following statements on the localization of ideals are easy and well-known. Analogous results hold for cones and we include their (entirely parallel) proofs~below.

\begin{lemma} \label{lemma:SatIdeal}
Let $A$ be an integral domain, $S \subseteq A$ a monoid and $\SR I \subseteq A$ an ideal.
\begin{enumerate}[label=(\arabic*)]
\item The extension of $\SR I$ in $S^{-1} A$ equals $S^{-1} \SR I \defas \Set{ \sfrac{f}{s} : f \in \SR I, s \in S }$. This ideal is proper if and only if $\SR I \cap S = \emptyset$.

\item The contraction $A \cap S^{-1} \SR I$ coincides with the saturation $\SR I:S$. It is proper if and only if $\SR I \cap S = \emptyset$.

\item If $\SR I$ is prime and $\SR I \cap S = \emptyset$, then $\SR I$ is saturated at $S$ and $S^{-1} \SR I$ is prime as well.

\item If $\SR I$ is real, then $S^{-1} \SR I$ is also real.
\end{enumerate}
\end{lemma}

We now turn to the effects of localization on cones. Again it turns out that saturations describe the contraction of the extension into a localized ring. These results are elementary but we prove them here for lack of a suitable reference.

\begin{lemma} \label{lemma:SatCone}
Let $A$ be an integral domain, $S \subseteq A$ a monoid and $\SR P \subseteq A$ a cone.
\begin{enumerate}[label=(\arabic*)]
\item \label{lemma:SatCone:1}
The extension $\cone_{\SOS{(S^{-1} A)}}{\SR P}$ of $\SR P$ in $S^{-1} A$ equals $S^{-2} \SR P \defas \Set{ \sfrac{p}{s^2} : p \in \SR P, s \in S }$. This cone is proper if and only if $S^2 \subseteq \SR P^+$.

\item \label{lemma:SatCone:2}
If $S \subseteq \SR P$, then $S^{-2} \SR P = S^{-1} \SR P$ and the contraction $A \cap S^{-1} \SR P$ coincides with the saturation $\SR P:S$. It is proper if and only if $S \subseteq \SR P^+$.

\item \label{lemma:SatCone:3}
If $\SR P$ is prime and $S \subseteq \SR P^+$, then $\SR P$ is saturated at $S$ and $S^{-2} \SR P$ is prime as well.
\end{enumerate}
\end{lemma}

\pagebreak %

\begin{proof}
\begin{paraenum}
\item[\ref{lemma:SatCone:1}]
Let $\SR P' = \cone_{\SOS{(S^{-1} A)}}{\SR P}$. The inclusion $S^{-2} \SR P \subseteq \SR P'$ is obvious. In the opposite direction, consider any element $a' = \sum_{i=1}^k \sigma_i p_i \in \SR P'$ where $\sigma_i \in \SOS{(S^{-1} A)}$ and $p_i \in \SR P$. Since $\sigma_i$ is a sum of squares of elements in $S^{-1} A$, there exist $s_i \in S$ which clear the denominators so that $s_i^2 \sigma_i \in \SOS{A}$. But this shows $\left(\prod_{i=1}^k s_i\right)^2 a' \in \cone_{\SOS{A}}{\SR P} = \SR P$ and thus $a' \in S^{-2} \SR P$. The cone is not proper if and only if $-1$ can be represented as a fraction $\sfrac{p}{s^2}$ with $p \in \SR P$ and $s \in S$. But this means $s^2 \in -\SR P$ and hence the cone is proper if and only if $S^2 \cap -\SR P = \emptyset$. Since $S^2 \subseteq \SR P$ always holds, this is equivalent to $S^2 \subseteq \SR P^+$.

\item[\ref{lemma:SatCone:2}]
If $S \subseteq \SR P$, then any $\sfrac{p}{s} \in S^{-1} \SR P$ equals $\sfrac{sp}{s^2} \in S^{-2} \SR P$; the other inclusion is obvious since $S$ is closed under multiplication.
To show that $\SR P:S$ is the contraction of $S^{-1} \SR P$, let $a = \sfrac{p}{s} \in A \cap S^{-1} \SR P$. Then $sa = p \in \SR P$ and this already shows $a \in \SR P:S$. The reverse inclusion follows by reading this argument in reverse. The saturation is proper if and only if $S \cap -\SR P = \emptyset$. Under the assumption that $S \subseteq \SR P$, this is equivalent to $S \subseteq \SR P^+$.

\item[\ref{lemma:SatCone:3}]
To see that $\SR P$ is saturated at $S$, take any $a \in \SR P:S$. There is $s \in S$ such that $sa \in \SR P$. Since $\SR P$ is a prime cone, this means $-s \in \SR P$ or $a \in \SR P$. Since $S \subseteq \SR P^+$ by assumption, $a \in \SR P$ follows~as~required.

The cone $\SR P$ is prime if and only if $\SR P = \pos \phi \defas \Set{ a \in A : \phi(a) \ge 0 }$ for some homomorphism $\phi \colon A \to K$ into an ordered field $K$, by~\cite[Proposition~4.3.4]{RealAlgebra}.
The assumption $S \subseteq \SR P^+$ means that $\phi(s) > 0$ for all $s \in S$. Since $K$ is an ordered field, $\phi(S)$ consists entirely of units and hence $\phi$ extends to a homomorphism $\ol{\phi} \colon S^{-1} A \to K$ via $\ol{\phi}(\sfrac{a}{s}) = \sfrac{\phi(a)}{\phi(s)}$. It is clear that $\ol{\phi}(\sfrac{a}{s}) \ge 0$ if and only if $a \in \SR P$, so $\pos \ol{\phi} = S^{-1} \SR P = S^{-2} \SR P$ is a prime cone.
\qedhere
\end{paraenum}
\end{proof}

Monoids play two distinct roles in our setting. First, their elements appear as denominators of rational maps and in order to write down their pullbacks, we have to localize. Second, the non-vanishing of polynomials is a valid and useful type of constraint on a statistical model to exclude degenerate loci or to turn weak inequalities into strict ones. This is modeled by monoids as well. To distinguish these two roles, we reserve capital letters like $S$ in the standard font for monoids defining localizations and denote monoids corresponding to non-vanishing model constraints by $\SR U$, using the same font as for ideals and cones.

A monoid $\SR U \subseteq A$ remains a monoid in $S^{-1} A$ but from the point of view of non-vanishing, it lacks the newly created units $\sfrac1s$. Therefore, the extension of $\SR U$ in $S^{-1} A$ is the monoid $S^{-1} \SR U = \Set{ \sfrac{u}{s} : u \in \SR U, s \in S }$. We use $S^\pm$ as a short-hand for $S^{-1} S$. We say that $\SR U$ is \emph{prime} if $ab \in \SR U$ implies $a, b \in \SR U$ for all $a, b \in A$. This is also sometimes called \emph{saturatedness} (without the reference to another monoid~$S$). The proof of the following simple facts is left to the reader.

\begin{lemma} \label{lemma:SatMonoid}
Let $A$ be an integral domain and $S, \SR U \subseteq A$ monoids. The contraction $A \cap S^{-1} \SR U$ coincides with the saturation $\SR U : S$. If $\SR U$ is prime, then $\SR U$ is saturated at $S$. If additionally $S \subseteq \SR U$, then $S^{-1} \SR U$ is prime as~well.
\end{lemma}

\subsection{The Positivstellensatz}
\label{sec: Positivstellen}

Let $f_i, p_j \in \BB R[x_1, \dots, x_n]$, for $i \in [r]$ and $j \in [s]$, be polynomials.
They define a \emph{basic closed} semialgebraic set $\Set{ x \in \BB R^n : f_i(x) = 0, \; p_j(x) \ge 0 }$. It is effectively given by the ideal $\SR I$ generated by the $f_i$ and the cone $\SR P$ generated by the $p_j$. These polynomials also define a \emph{basic open} semialgebraic set $\Set{ x \in \BB R^n : f_i(x) = 0, \; p_j(x) > 0 }$ where the weak inequalities are made strict. This set is formally given by $\SR I$ and $\SR P$ as before together with the monoid $\SR U$ generated by the $p_j$.

Unless the ideal $\SR I$ is radical, it does not contain all polynomials which vanish on the complex variety cut out by the polynomial equations $f_i = 0$. Over the real numbers, which are not algebraically closed, $\SR I$ may be even further from the vanishing ideal. Similarly, the cone $\SR P$ does not in general account for all non-negativities on the semialgebraic set. A complete characterization of the vanishing ideal and the non-negative cone are provided by the Positivstellensatz.

\begin{theorem}[{Formal Positivstellensatz \cite[Proposition~4.4.1]{RealAlgebra}}]
Let $\SR I, \SR P, \SR U$ be an~ideal, a~cone and a~monoid, respectively, in a ring~$A$. The following two statements are equivalent:
\begin{enumerate}[label=(\alph*)]
\item \label{thm:Positivstellen:1}
There does not exist a homomorphism $\phi \colon A \to K$ into an ordered field such that $\phi(f) = 0$ for all $f \in \SR I$, $\phi(p) \ge 0$ for all $p \in \SR P$ and $\phi(u) \not= 0$ for all $u \in \SR U$.
\item \label{thm:Positivstellen:2}
There exist $f \in \SR I$, $p \in \SR P$ and $u \in \SR U$ such that $f = p + u^2$, i.e., $\SR I \cap (\SR P + \SR U^2) \not= \emptyset$.
\end{enumerate}
\end{theorem}

Take $A = \BB R[x_1, \dots, x_n]$ and suppose that $\SR I$, $\SR P$ and $\SR U$ are finitely generated. They define a semialgebraic set~$\Theta \subseteq \BB R^n$. Condition~\ref{thm:Positivstellen:1} applied to all evaluation homomorphisms $A \to \BB R$ implies that the equality, non-negativity and non-vanishing conditions are incompatible over $\BB R$; in other words, the semialgebraic set $\Theta$ is empty.
Condition~\ref{thm:Positivstellen:2} then ensures the existence of a certificate for this emptiness, for if $f = p + u^2$ holds in~$A$, then it holds under every homomorphism $\phi$ into $\BB R$. But on every point $x \in \Theta$, the left-hand side $f(x) = 0$ while the right-hand side $p(x) + u(x)^2 > 0$. This shows that there cannot be a point in~$\Theta$.

The Positivstellensatz shows that the ideal, cone and monoid defining $\Theta$ geometrically contain enough information to decide whether it is empty. To check whether a polynomial $F$ vanishes on~$\Theta$, one simply adds $F$ to the monoid $\SR U$, thus defining the subset $\Theta'$ of points in $\Theta$ on which $F$ does not vanish. The function~$F$ vanishes on $\Theta$ if and only if $\Theta' = \emptyset$ and the certificate for emptiness from condition~\ref{thm:Positivstellen:2} can be rewritten to derive a representation of~$F$. This technique can be used to characterize the vanishing, non-negative and positive polynomials on $\Theta$ in terms of its ideal, cone and monoid.

\begin{corollary}[Positivstellensatz] \label{cor:Positivstellen}
Let $\SR I$, $\SR P$ and $\SR U$ be an ideal, a cone and a monoid in $A = \BB R[x_1, \dots, x_n]$. Let $\Theta \subseteq \BB R^n$ be the set they define. For any $F \in A$:
\begin{enumerate}[label=(\arabic*), itemsep=0.3em, parsep=0pt]
\item $F = 0$ on $\Theta$ if and only if $\exists m \ge 0, u \in \SR U \colon -F^{2m} u^2 \in \SR I + \SR P$.
\item $F \ge 0$ on $\Theta$ if and only if $\exists m \ge 0, p \in \SR P, u \in \SR U \colon -F^{2m} u^2 + F p \in \SR I + \SR P$.
\item $F > 0$ on $\Theta$ if and only if $\exists p \in \SR P: F p \in \SR I + \SR P + \SR U^2$.
\end{enumerate}
\end{corollary}

\section{Birational implicitization}
\label{sec: biri}

We now use the preliminaries on localizations from \Cref{sec: localization} to compute the images of ideals, cones and monoids under a birational isomorphism. By the Positivstellensatz in \Cref{sec: Positivstellen}, this is enough to characterize all vanishing, non-negative and positive polynomials on the image of a rational map. This yields our solution to the implicitization problem in \Cref{thm:Biri}.

The~supporting lemmas hold in great generality, only supposing that the ambient rings are integral domains. In~particular, the ideals, cones and monoids we consider need not be finitely generated. However, our proofs preserve finite generation and they yield the most practical results in this case. With applications in mind, it is instructive to pretend that all rings are polynomial rings with finitely generated ideals, cones and monoids.

\subsection{Full birational isomorphisms}

We work in a geometric setting where two integral domains $A$ and $\<A$ have isomorphic localizations $S^{-1} A \cong \<S^{-1} \<A$. If $A$ and $\<A$ are the coordinate rings of affine varieties $\CC V$ and $\<{\CC V}$, respectively, then this setup amounts to saying that $\CC V$ and $\<{\CC V}$ are birational. We spell out the localization explicitly because the localizing monoids $S$ and $\<S$ play a role in the results we are after, and more control over them, e.g., knowing a finite generating set, enables stronger conclusions.

\begin{definition}
Let $A$ and $\<A$ be integral domains and $S \subseteq A$ and $\<S \subseteq \<A$ monoids such that the localizations $S^{-1} A$ and $\<S^{-1} \<A$ are isomorphic. Thus we have the following diagram:
\begin{center}
\begin{tikzcd}
  A  \arrow[r, hook]  & S^{-1}   A \arrow[r, "\phi", shift left] &
\<S^{-1} \<A \arrow[l, "\psi", shift left] & \<A \arrow[l, hook']
\end{tikzcd}
\end{center}
We refer to this situation as a \emph{birational isomorphism} between $A$ and $\<A$ and denote it more concisely by $\phi\colon S^{-1} A \iso \<S^{-1} \<A \revcolon \psi$. The isomorphism is \emph{full} if $\<S^\pm = \phi(S^\pm)$ (and thus $S^\pm = \psi(\<S^\pm)$).
\end{definition}

For convenience we prefer to work with full birational isomorphisms. Thus our first concern is to extend a given isomorphism to a full one.

\begin{lemma} \label{lemma:Extend}
Let $\phi\colon S^{-1} A \iso \<S^{-1} \<A \revcolon\psi$ be a birational isomorphism and $T = A \cap (S^\pm \cdot \psi(\<S^\pm))$ and $\<T = \<A \cap (\<S^\pm \cdot \phi(S^\pm))$. Then there exist homomorphisms $\ol\phi$ and $\ol\psi$ such that the following diagram commutes:
\begin{center}
\begin{tikzcd}
  A  \arrow[r, hook]  & S^{-1}   A \arrow[d, "\phi"', shift right] \arrow[r, hook] & T^{-1} A \arrow[d, "\ol{\phi}"', shift right] \\[1.3em]
\<A \arrow[r, hook] & \<S^{-1} \<A \arrow[u, "\psi"', shift right] \arrow[r, hook] & \<T^{-1} \<A \arrow[u, "\ol{\psi}"', shift right]
\end{tikzcd}
\end{center}
In particular, $\ol\phi: T^{-1} A \iso \<T^{-1} \<A \revcolon\ol\psi$ is a full birational isomorphism between $A$ and~$\<A$.
\end{lemma}

\begin{proof}
The inclusion $S \subseteq T$ induces a canonical inclusion $S^{-1} A \into T^{-1} A$. Since $\phi(T) \subseteq \phi(S^\pm) \cdot \<S^\pm \subseteq \<S^{-1} \<T$ is a subset of the units of $\<T^{-1} \<A$, the universal property of localization \cite[Proposition~3.1]{AtiyahMacdonald} extends the map $S^{-1} A \xto{$\phi$} \<S^{-1} \<A \into \<T^{-1} \<A$ canonically to $\ol{\phi}\colon T^{-1} A \to \<T^{-1} \<A$. The same applies to $\<T$ and $\psi$ giving the map $\ol\psi$. To see that $\ol\phi$ and $\ol\psi$ are mutually inverse, it~suffices to compute
\[
  \ol\psi\left(\ol\phi\left(\frac{a}{t}\right)\right) = \ol\psi\left(\frac{\phi(a)}{\phi(t)}\right) = \frac{\psi(\phi(a))}{\psi(\phi(t))} = \frac{a}{t}.
\]
This extends the isomorphism given by $\phi$ and $\psi$ to a birational isomorphism $\ol\phi\colon T^{-1} A \iso \<T^{-1} \<A \revcolon\ol\psi$. It was already shown above that $\phi(T) \subseteq \<S^{-1} \<T \subseteq \<T^\pm$, and $\psi(\<T) \subseteq T^\pm$ follows analogously. By~introducing denominators and using that $\ol\phi^{-1} = \ol\psi$, these two inclusions become $\ol\phi(T^\pm) = \<T^\pm$. Hence the isomorphism is full.
\end{proof}

Under mild additional assumptions, the extension constructed in \Cref{lemma:Extend} even preserves finite generation of the localizing monoids.

\begin{lemma} \label{lemma:ExtendFinite}
Let $\phi\colon S^{-1} A \iso \<S^{-1} \<A \revcolon\psi$ be a birational isomorphism between unique factorization domains $A$ and $\<A$ and suppose that $S$ and $\<S$ are finitely generated by irreducible elements. Then $T = A \cap (S^\pm \cdot \psi(\<S^\pm))$ and $\<T = \<A \cap (\<S^\pm \cdot \phi(S^\pm))$ are finitely generated and the canonically extended birational isomorphism $\ol{\phi}\colon T^{-1} A \iso \<T^{-1} \<A \revcolon\ol{\psi}$ is full.
\end{lemma}

\begin{proof}
Given \Cref{lemma:Extend}, we only need to proof that $T$ is finitely generated (the proof for $\<T$ is symmetric). Suppose that $S = \monoid{s_\ell : \ell \in L}$ and $\<S = \monoid{\<s_\ell : \ell \in L}$ are generated by irreducibles. (Without loss of generality, we may suppose that a single finite set $L$ indexes both sets of generators.)
It is easy to infer from the irreducibility of $\<s_\ell$ and the assumption that $S$ and $\<S$ are generated by irreducibles that $\num \psi(\<s_\ell) = s_\ell' s$ for some $s \in S$ and an element $s_\ell' \in A$ which is irreducible or a unit.
We claim that $T = U \defas \monoid{s_\ell, s_\ell' : \ell \in L}$.
One inclusion is clear since $s_\ell, s_\ell' \in T$. For~the other inclusion, notice that $S^{-1} U = S^\pm \cdot \psi(\<S^\pm)$. \Cref{lemma:SatMonoid} shows that the contraction is $T = A \cap S^{-1} U = U : S$, so it remains to show that $U$ is saturated at~$S$. Take any $a \in A$ with $sa = u \in U$. Since $A$ is a UFD, $u$ can be written (uniquely, up to a unit of $A$) as a product of irreducibles. Since $u$ is a product of the generators of $U$, which are all irreducible, its irreducible decomposition can be assumed to have all its factors in~$U$. Removing the irreducible factors corresponding to $s \in S \subseteq U$ writes $a$ as a product of generators of~$U$, hence~$a \in U$.
\end{proof}

\subsection{Transferring ideals, cones and monoids}

We are now ready to prove the lemmas supporting our main theorem. The idea is always the same: given an ideal, cone or monoid $\<{\SR X}$ in $\<A$, it can be extended to $\<S^{-1} \<A$, sent through the isomorphism to $S^{-1} A$ and then contracted to~$A$ yielding a corresponding ideal, cone or monoid~$\SR X$. We give generators of $\SR X$ in terms of generators of $\<{\SR X}$ up to saturation at~$S$. For example, if the birational isomorphism is established by the pullback of a rational parametrization map and $\<{\SR X}$ is the vanishing ideal of the parameter space, then $\SR X$ will be the vanishing ideal of the image.

\begin{lemma} \label{thm:Ideal}
Let $\phi \colon S^{-1} A \iso \<S^{-1} \<A \revcolon \psi$ be a full birational isomorphism.
Consider an ideal $\<{\SR I} = \ideal_{\<A}{\<f_i : i \in I}$ with $\<{\SR I} \cap \<S = \emptyset$, where $I$ is some index set. Let $\SR I = A \cap \psi(\<S^{-1} \<{\SR I})$ be the corresponding ideal in $A$ and $\SR J = \ideal_A{ \num \psi(\<f_i) : i \in I }$ the ideal of numerators.  Then:
\begin{inparaenum}
\item\label{thm:Ideal:1}$\SR I \cap S = \emptyset$.
\item\label{thm:Ideal:2}The~chain of inclusions $\SR J \subseteq \SR I \subseteq \SR J:S = \SR I:S$ holds.
\item\label{thm:Ideal:3}If~$\<{\SR I}$ is prime or real, then so is~$\SR I$.
\item\label{thm:Ideal:4}If~$\<{\SR I}$ is saturated at $\<S$, then $\SR I$ is saturated at~$S$.
In this case $\SR I = \SR J:S$.
\end{inparaenum}
\end{lemma}

\begin{proof}
\begin{paraenum}
\item[\ref{thm:Ideal:1}]
Take any element $s \in \SR I \cap S$. Then $\phi(s) \in \<S^{-1} \<{\SR I} \cap \<S^\pm$. But then by clearing denominators there exists $\<s \in \<S$ such that $\<s \phi(s) \in \<{\SR I} \cap \<S$, a contradiction.

\item[\ref{thm:Ideal:2}]
By definition $\phi(f_i) = \phi(s_i) \<f_i$ for some $s_i \in S$. Since $\phi(S^\pm) = \<S^\pm$, this shows $\phi(f_i) \in \<S^{-1} \<{\SR I}$ for each $i$ and thus $\SR J \subseteq \SR I$. This implies $\SR J : S \subseteq \SR I : S$.
For the middle inclusion let $f \in \SR I$ be arbitrary. Since $\phi(f) \in \<S^{-1} \<{\SR I}$, there exist a finite subset $I' \subseteq I$ and $\<h_i \in \<S^{-1} \<A$ such that $\phi(f) = \sum_{i\in I'} \<h_i \<f_i$. The images $\psi(\<h_i \<f_i)$ belong to $S^{-1} A$ and since there are only finitely many such terms, there exists a common denominator $s \in S$ of all of these fractions so that $s \psi(\<h_i \<f_i) = s \psi(\<h_i) \sfrac{f_i}{s_i} = h_i f_i$ with $h_i \in A$. Then the computation
\[
  s f = s \psi(\phi(f)) = \sum_{i \in I'} s \psi(\<h_i \<f_i) = \sum_{i \in I'} h_i f_i \in \SR J,
\]
shows $f \in \SR J:S$ and thus $\SR I \subseteq \SR J:S$. Saturating the chain of inclusions $\SR I \subseteq \SR J:S \subseteq \SR I:S$ at $S$ again yields $\SR I:S = \SR J:S$.

\item[\ref{thm:Ideal:3}]
Suppose that $\<{\SR I}$ is real or prime. Since it is disjoint from $\<S$, the extension $\<S^{-1} \<{\SR I}$ in the localized ring is real, respectively prime, by \Cref{lemma:SatIdeal}. Both properties are preserved by the isomorphism $\psi$ and the contraction into~$A$ (which is the preimage of the canonical inclusion map $A \into S^{-1} A$).

\item[\ref{thm:Ideal:4}]
Now suppose that $\<{\SR I}$ is saturated at~$\<S$ and take any $sa \in \SR I$ for $a \in A$ and $s \in S$. By~definition of~$\SR I$, there exists $\<s \in \<S$ such that $\<s \phi(sa) \in \<{\SR I}$. Since $\phi(sa) = \phi(s) \phi(a)$ with $\phi(s) \in \<S^\pm$, there is another $\<s' \in \<S$ such that $\<s \<s' \phi(a) \in \<{\SR I}$. By assumption, $\<{\SR I}$ is saturated at $\<S$, so $\phi(a) \in \<{\SR I}$ and thus $a \in \SR I$. This proves that $\SR I$ is saturated at~$S$ and thus the chain $\SR I \subseteq \SR J : S \subseteq \SR I:S = \SR I$ collapses proving the equality $\SR I = \SR J : S$.
\qedhere
\end{paraenum}
\end{proof}

\begin{remark}
A version of this lemma appeared as \cite[Lemma~4.8]{ColoredDAG} in the context of colored Gaussian DAG models. Its formulation there did not require $\phi$ to be an isomorphism of localized rings; it only required a left-inverse~$\psi$. Under the additional assumption that $\phi$ is surjective, \cite[Lemma~4.8]{ColoredDAG} turns into a special case of \Cref{thm:Ideal}.
\end{remark}

\begin{lemma} \label{thm:Cone}
Let $\phi \colon S^{-1} A \iso \<S^{-1} \<A \revcolon \psi$ be a full birational isomorphism.
Consider a cone $\<{\SR P} = \cone_{\SOS{\<A}}{\<p_i : i \in I}$ with $\<S \subseteq \<{\SR P}^+$, where $I$ is some index set. Let $\SR P = A \cap \psi(\<S^{-2} \<{\SR P})$ be the corresponding cone in $A$ and $\SR Q = \cone_{\SOS{A}}{\num \psi(\<p_i) : i \in I}$ the cone of numerators.  Then:
\begin{inparaenum}
\item\label{thm:Cone:1}$S \subseteq \SR P^+$.
\item\label{thm:Cone:2}The inclusions $\SR Q \subseteq \SR P \subseteq \cone_{\SR Q}{S} : S = \SR P : S$ hold.
\item\label{thm:Cone:3}If~$\<{\SR P}$ is prime, then so is $\SR P$.
\item\label{thm:Cone:4}If~$\<{\SR P}$ is saturated at $\<S$, then $\SR P$ is saturated at~$S$.
In this case $\SR P = \cone_{\SR Q}{S}:S$.
\end{inparaenum}
\end{lemma}

\begin{proof}
\begin{paraenum}
\item[\ref{thm:Cone:1}]
First note that fullness and $\<S \subseteq \<{\SR P}$ imply $\phi(S) \subseteq \<S^\pm \subseteq \<S^{-2} \<{\SR P}$ and thus $S \subseteq \SR P$. If $s \in -\SR P \cap S$, then $\phi(s) \in -\phi(\SR P) \cap \<S^\pm \subseteq -\<S^{-2} \SR P \cap \<S^\pm$. Thus there exists $\<s \in \<S$ such that $\<s \phi(s) \in -\<{\SR P} \cap \<S$, a contradiction.

\item[\ref{thm:Cone:2}]
Consider any generator $p_i$ of $\SR Q$. By definition $\phi(p_i) = \phi(s_i) \<p_i$ for some $s_i \in S$, but since $\phi(s_i) \in \<S^\pm \subseteq \<S^{-2} \<{\SR P}$, this shows $p_i \in \SR P$ and thus $\SR Q \subseteq \SR P$.
For the middle inclusion let $p \in \SR P$ be arbitrary. Since~$\phi(p) \in \<S^{-2} \<{\SR P}$, there exist a finite subset $I' \subseteq I$, an element $\<s \in \<S$ and $\<\sigma_K \in \SOS{\<A}$, for $K \subseteq I'$, such that $\phi(p) = \frac{1}{\<s^2} \sum_{K \subseteq \in I'} \<\sigma_K \<p^K$, where we use monomial-subset notation $\<p^K \defas \prod_{k \in K} \<p_k$. There is a common denominator $s \in S$ for the finitely many fractions $\psi(\frac{\<\sigma_K}{\<s^2} \<p^K) \in S^{-1} A$ so that
\[
  s \psi\left(\frac{\<\sigma_K}{\<s^2} \<p^K\right) = s \psi\left(\frac{\<\sigma_K}{\<s^2}\right) \prod_{k \in K} \frac{p_k}{s_k} = \sigma_K s_K p^K,
\]
for some $\sigma_K \in \SOS{A}$ and $s_K \in S$. Note that unlike in the ideal case \Cref{thm:Ideal}, multiplying through by the common denominator $s$ leaves factors $s_K \in S$ in the above expression which cannot necessarily be absorbed into the coefficient~$\sigma_K$ or the product of generators $p^K$. However, this expression shows that $sp = s \psi(\phi(p)) \in \cone_{\SR Q}{S}$ and thus $p \in \cone_{\SR Q}{S}:S$ as required.
This proves $\SR Q \subseteq \SR P \subseteq \cone_{\SR Q}{S}:S$. Adjoining $S$ to every cone in this chain (using that $S \subseteq \SR P$ so that $\cone_{\SR P}{S} = \SR P$) and then saturating at $S$ proves $\cone_{\SR Q}{S}:S = \SR P:S$.

\item[\ref{thm:Cone:3} and \ref{thm:Cone:4}] are entirely analogous to \Cref{thm:Ideal}.
\qedhere
\end{paraenum}
\end{proof}

\begin{lemma} \label{thm:Monoid}
Let $\phi \colon S^{-1} A \iso \<S^{-1} \<A \revcolon \psi$ be a full birational isomorphism.
Consider a monoid $\<{\SR U} = \monoid{\<u_i : i \in I}$ with $\<S \subseteq \<{\SR U}$, where $I$ is some index set. Let $\SR U = A \cap \psi(\<S^{-1} \<{\SR U})$ be the corresponding monoid in $A$ and $\SR V = \monoid{\num \psi(\<u_i) : i \in I}$ the monoid of numerators. Then:
\begin{inparaenum}
\item\label{thm:Monoid:1}$S \subseteq \SR U$.
\item\label{thm:Monoid:2}The inclusions $\SR V \subseteq \SR U \subseteq S\SR V : S = \SR U : S$ hold.
\item\label{thm:Monoid:3}If~$\<{\SR U}$ is prime, then so is $\SR U$.
\item\label{thm:Monoid:4}If~$\<{\SR U}$ is saturated at $\<S$, then $\SR U$ is saturated at~$S$.
In this case $\SR U = S\SR V : S$.
\end{inparaenum}
\end{lemma}

\begin{proof}
The proof of \Cref{thm:Cone} is easily modified to deal with a monoid instead of a cone.
\end{proof}

\subsection{Birational implicitization}

We now return to the more concrete setting of \Cref{sec: intro}. Consider a rational map $\alpha\colon \BB R^n \ratto \BB R^m$ with rational inverse~$\beta$. Being birational, the domain and codomain of $\alpha$ necessarily have the same dimension~$m=n$. This induces a specific type of birational isomorphism between the coordinate algebras $A = \BB R[x_1, \dots, x_n]$ of the codomain and $\<A = \BB R[\<x_1, \dots, \<x_n]$ of the domain of~$\alpha$.

\begin{definition}
A full birational isomorphism $\phi\colon S^{-1} A \iso \<S^{-1} \<A \revcolon\psi$ is \emph{affine} if $A = \BB R[x_1, \dots, x_n]$ and $\<A = \BB R[\<x_1, \dots, \<x_n]$ and $\phi$ and $\psi$ are $\BB R$-algebra isomorphisms.
\end{definition}

\begin{remark}
Some assumptions simplify in the affine setting:
\begin{paraenum}[label=(\roman*)]
\item
The localizing monoids $S$ and $\<S$ are finitely generated by the denominators of $\beta^*(\<x_i)$ and $\alpha^*(x_i)$, respectively. Since polynomial rings are UFDs, we may replace each of these generators by the (still finitely many) factors in their irreducible decompositions. This enlarges the monoid in a geometrically sound way: if a polynomial is supposed to vanish nowhere on a given space, then its factors must also vanish nowhere. By virtue of \Cref{lemma:ExtendFinite}, we can further enlarge these monoids to make the isomorphism~full without sacrificing finite generation. This is a prerequisite for computing the saturations from \Cref{thm:Ideal,thm:Cone,thm:Monoid} in practice.

\item
Let $\<{\CC M}$ be the semialgebraic set described by an ideal $\<{\SR I}$, a cone $\<{\SR P}$ and a monoid $\<{\SR U}$ in a polynomial ring and let $\<S$ be a localizing set. To apply \Cref{thm:Ideal,thm:Cone,thm:Monoid}, it suffices to make the following assumptions:
\begin{inparaenum}
\item $\<S \subseteq \<{\SR U}$,
\item $\<S \subseteq \<{\SR P}$,
\item $\<{\SR I} \cap \<{\SR U} = \emptyset$, and
\item $|\<{\SR P}| \subseteq \<{\SR I}$.
\end{inparaenum}
The~first two correspond to the non-trivial assumption that ``denominators'' in $\<S$ are positive on~$\<{\CC M}$. The~third assumption is necessary for $\<{\CC M}$ to be non-empty. As for the last assumption, any equation which is implied by the cone can also be enforced through the ideal without changing the set $\<{\CC M}$. These~assumptions guarantee that if $\<{\CC M} \not= \emptyset$, then $\<{\SR I} \cap \<S = \emptyset$ and $\<S \subseteq \<{\SR P}^+$.

\item
Moreover, if a generator $\<s$ of $\<S$ lies in $\<{\SR P}^- = -\<{\SR P} \setminus \<{\SR P}$, then it may be replaced by $-\<s \in \<{\SR P}^+$, generating a different monoid which however yields an isomorphic localization (and an equivalent inequation through $\<{\SR U}$). In particular, if $\<{\CC M}$ is \emph{connected} and $s$ vanishes nowhere (according to~$\<s \in \<{\SR U}$), then either $\<s$ or $-\<s$ is positive on this set and may be added to $\<{\SR P}$ without changing the geometry of~$\<{\CC M}$.
\end{paraenum}
\end{remark}

\begin{theorem} \label{thm:Biri}
Let $\phi\colon S^{-1} A \iso \<S^{-1} \<A \revcolon\psi$ be an affine birational isomorphism induced by the rational map $\alpha\colon \BB R^n \ratto \BB R^n$.
Suppose that $S = \monoid{s_\ell : \ell \in L}$ and $\<S = \monoid{\<s_\ell : \ell \in L}$. Let $I, J, K$ be further index sets and $\<f_i, \<p_j, \<u_k \in \BB R[\<x_1, \dots, \<x_n]$ be collections of polynomials, for all $i \in I$, $j \in J$, $k \in K$. They define a parameter space
\begin{equation*}
  \<{\CC M} = \Set{ \<x \in \BB R^n : \<f_i(\<x) = 0, \;\; \<p_j(\<x) \ge 0, \;\; \<u_k(\<x) \not= 0, \;\; \<s_\ell(x) > 0 }.
\end{equation*}
Suppose $\<{\CC M} \not= \emptyset$. With $f_i = \num \psi(\<f_i)$, $p_j = \num \psi(\<p_j)$ and $u_k = \num \psi(\<u_k)$, the image equals
\begin{equation}
  \label{eq:Markov} \tag{$*$}
  \alpha(\<{\CC M}) = \Set{ x \in \BB R^n : f_i(x) = 0, \;\; p_j(x) \ge 0, \;\; u_k(x) \not= 0, \;\; s_\ell(x) > 0 }.
\end{equation}
In particular, if $I, J, K, L$ are all finite, then the description of~$\alpha(\<{\CC M})$ is finite.
\end{theorem}

\begin{proof}
Let $\CC M = \alpha(\<{\CC M})$ and denote the set on the right-hand side of~\eqref{eq:Markov} by~$\CC M'$. We first prove $\CC M \subseteq \CC M'$. By definition of $\<{\CC M}$, all rational functions in $\<S^\pm$ are positive on it. By fullness of the birational isomorphism, every $s \in S$ can be represented as $s = \psi(\sfrac{\<s}{\<s'})$ and for every $\<x \in \<{\CC M}$:
\[
  s(\alpha(\<x)) = \phi(s)(\<x) = \frac{\<s(\<x)}{\<s'(\<x)} > 0.
\]
Thus all polynomials in $S$ are positive on~$\CC M$. Now consider a polynomial $a = \num \psi(\<a)$. There exists $s \in S$ such that $a = s \psi(\<a)$. For any $\<x \in \<{\CC M}$, the pullback rewrites $a(\alpha(\<x)) = s(\alpha(\<x)) \cdot \<a(\<x)$. Hence $a$ vanishes, is non-negative or non-vanishing on $\CC M$ if and only if $\<a$ has the respective property with respect to~$\<{\CC M}$. This argument shows that every point in $\CC M$ satisfies the constraints defining $\CC M'$, hence~$\CC M \subseteq \CC M'$.

To prove the reverse inclusion, recall that $\CC M' \subseteq \CC M$ if and only if every positive polynomial on $\CC M$ is also positive on $\CC M'$. The non-obvious direction follows by assuming that there exists $x_0 \in \CC M' \setminus \CC M$ and considering the polynomial $\lVert x-x_0\rVert^2$ which is positive on $\CC M$ but has a zero on~$\CC M'$. Now consider a polynomial $F$ which is positive on $\CC M$. By definition, this means that $\phi(F)$ is positive on $\<{\CC M}$. The Positivstellensatz (\Cref{cor:Positivstellen}) shows that there exist $\<f \in \<{\SR I}$, $\<p, \<q \in \<{\SR P}$ and $\<u \in \<{\SR U}$ such that $\phi(F) \<p = \<f + \<q + \<u^2$, where $\<{\SR I} = \ideal_{\<A}{f_I}$, $\<{\SR P} = \cone_{\SOS{\<A}}{\<p_J, \<s_L}$ and $\<{\SR U} = \monoid{\<u_K, \<s_L}$. Applying $\psi$ to this equation yields
\begin{equation}
  \label{eq:Fpsi}
  F\psi(\<p) = \psi(\<f) + \psi(\<q) + \psi(\<u)^2.
\end{equation}
Multiplying through by any common denominator $s \in S$ of the fractions in this equation gives $s\psi(\<p) \in A$ but also $s\psi(\<p) = \psi(\frac{\<s}{\<s'} \<p) \in \psi(\<S^{-1} \<{\SR P})$ by the fullness of the birational isomorphism and $\<S \subseteq \<{\SR P}$. Thus $s\psi(\<p) \in \SR P \defas A \cap \psi(\<S^{-1} \<{\SR P})$. Since $\<{\CC M}$ is non-empty, we have $\<S \subseteq \<{\SR P}^+$ and thus \Cref{thm:Cone} applies to compute $\SR P : S = \cone_{\SOS{A}}{p_J, s_L} : S$.
For~a~large enough $s \in S$ (with respect to the divisibility ordering), this even proves that $s\psi(\<p) \in \cone_{\SOS{A}}{p_J, s_L}$. The~same arguments apply to the ideal and monoid, invoking \Cref{thm:Ideal,thm:Monoid}, respectively. Continuing with \eqref{eq:Fpsi}, this~yields
\begin{equation}
  \label{eq:Fp}
  Fp = f + q + u^2,
\end{equation}
with $f \in \ideal_{A}{f_I}$, $p, q \in \cone_{\SOS{A}}{p_J, s_L}$ and $u \in \monoid{u_K, s_L}$. But then $F$ is clearly positive on~$\CC M'$. This~completes the proof of $\CC M' \subseteq \CC M$ and hence of their equality.
\end{proof}

We now shift the priorities slightly. In order to use \Cref{thm:Cone} to transfer inequalities, \Cref{thm:Biri} assumes that the generators of the localizing monoid are positive on the parameter space. If inequalities are not of interest, then this assumption can be weakened. If we assume, on the other hand, that the vanishing ideal of the parameter space is known and prime, we obtain the vanishing ideal of the model up to a saturation, instead of just a set of equations which cut it out geometrically.

\begin{theorem} \label{thm:VanishingIdeal}
Let $\phi\colon S^{-1} A \iso \<S^{-1} \<A \revcolon\psi$ be an affine birational isomorphism induced by the rational map $\alpha\colon \BB R^n \ratto \BB R^n$.
Let $\<{\CC M}$ be an irreducible algebraic variety and $\<{\SR I} = \ideal_{\<A}{\<f_i : i \in I}$ its vanishing ideal. If $\<{\SR I} \cap \<S = \emptyset$, then the vanishing ideal of $\CC M = \alpha(\<{\CC M})$ is~$\ideal_{A}{\num \psi(\<f_i) : i \in I} : S$.
\end{theorem}

\begin{proof}
We claim that $\SR I = A \cap \psi(\<S^{-1} \<{\SR I})$ is the vanishing ideal of~$\CC M$. Indeed, a polynomial $f \in A$ vanishes on $\CC M$ if and only if $\alpha^*(f) \in \<S^{-1} \<{\SR I}$, where $\alpha^*\colon A \to \<S^{-1} \<A$ is the pullback of the parametrization. In the birational isomorphism, $\alpha^*$ is canonically extended to $\phi\colon S^{-1} A \to \<S^{-1} \<A$ so that $\alpha^*(f) = \phi(f)$. Hence the previous condition can be equivalently written as $\phi(f) \in \<S^{-1} \<{\SR I}$. Applying the inverse $\psi$ gives $f \in A \cap \psi(\<S^{-1} \<{\SR I}) = \SR I$.
The statement then follows from \Cref{thm:Ideal}: since $\<{\SR I}$ is a real prime ideal, so is $\SR I$, hence it is saturated at~$S$, which proves $\SR I = \SR I : S = \ideal_{A}{\num \psi(\<f_i) : i \in I} : S$.
\end{proof}

\section{Ambirational statistical models}
\label{sec:Ambirational}

We call a parametrized statistical model $\CC M = \alpha(\<{\CC M})$ \emph{ambirational} if its parametrization map~$\alpha$ induces an affine birational isomorphism between the ambient affine spaces in which its parameter space $\<{\CC M}$ and the model $\CC M$ itself are embedded, such that the elements of the localizing monoid $\<S$ vanish nowhere on~$\<{\CC M}$. In particular, $\CC M$ is globally rationally identifiable. This section identifies a collection of commonly used statistical models as ambirational models. Our methods open up the possibility of studying submodels of these classical statistical models obtained by imposing expert knowledge in the form of additional equations and inequalities on the parameters. %

\subsection{Linear concentration models}

Consider the affine space $\Sym_n$ of symmetric $n \times n$ matrices. The subset of positive definite matrices $\PD_n$ is a full-dimensional convex cone. Any linear subspace $\CC L \subseteq \Sym_n$ defines a \emph{spectrahedron} $\CC L \cap \PD_n$. These sets appear as feasible sets in semidefinite optimization. We may view their elements as the covariance matrices of multivariate Gaussian distributions with mean zero and then $\CC L \cap \PD_n$ becomes a \emph{linear covariance model}; see~\cite{LinearCovariance} for prominent examples in this class. There is nothing to be done for these models from the point of view of implicitization. Instead, we view $\CC L \cap \PD_n$ as a set of concentration matrices of mean-zero Gaussian distributions and study the associated model in covariance coordinates $\CC M(\CC L) \defas \Set{ \Sigma \in \PD_n : \Sigma^{-1} \in \CC L \cap \PD_n }$.

This model is ambirational. To see this, let $\alpha\colon \Sym_n \ratto \Sym_n$ be the matrix inversion map $K \mapsto \Sigma = K^{-1}$. This is a rational function by the well-known relation $|K| \Sigma = \Adj K$ between the inverse and the adjugate. The parameter identification map $\beta = \alpha^{-1}$ is given by matrix inversion as well.
On the algebraic level, the pullback of matrix inversion gives an isomorphism of localized polynomial rings $\phi = \alpha^* \colon |\Sigma|^{-1} \BB R[\Sigma] \to |K|^{-1} \BB R[K]$, where the notation $|\Sigma|^{-1} \BB R[\Sigma]$ is short for $\monoid{|\Sigma|}^{-1} \BB R[\Sigma]$. Since $\phi(|\Sigma|) = |K|^{-1}$, this is a full and hence affine birational isomorphism.

The parameter space $\CC L \cap \PD_n$ is described by linear polynomial equations $\<f_i$ generating the (vanishing) ideal of $\CC L$ and the positive definiteness constraints. The latter may be expressed by making the principal minors $|K_I|$, for all $I \subseteq [n]$, non-negative and non-vanishing. These constraints naturally include $|K| > 0$, as required by \Cref{thm:Biri}. Note that $\num \psi(|K_I|) = |\Sigma_{I^\co}|$ where $I^\co = [n] \setminus I$, so the positive definiteness constraints on $K$ get transformed into positive definiteness of~$\Sigma$. As a subset of $\PD_n$, the model $\CC M$ is described by the equations $\num \psi(\<f_i) = 0$.

\begin{example}[Undirected graphical models] \label{ex:Undirected}
Let $G = (V, E)$ be an undirected graph on $|V|=n$ vertices. With the linear space $\CC L(E) = \Set{ K = (\kappa_{ij}) \in \Sym_n : \text{$\kappa_{ij} = 0$ for $ij \not\in E$} }$, we obtain the usual undirected Gaussian graphical model $\CC M = \CC M(G)$. According to \Cref{thm:Biri}, this model is cut out of $\PD_n$ by the vanishing of the cofactors $\num \psi(\kappa_{ij}) = (-1)^{i+j}|\Sigma_{ij|V\setminus ij}|$, for $ij \not\in E$, where $\Sigma_{ij|C}$ is the submatrix with rows $\{i\} \cup C$ and columns $\{j\} \cup C$.
These polynomial equations are equivalent to the (maximal) conditional independence statements $\CI{i,j|V \setminus ij}$, for $ij \not\in E$. This recovers the classical \emph{pairwise Markov property} of undirected graphical models \cite[Section~13.1]{Sullivant}. Moreover, since the equations defining $\CC L(E)$ generate a real prime ideal~$\ideal_{\BB R[K]}{\kappa_{ij} : ij \not\in E}$ of the same dimension as $\CC L(E) \cap \PD_n$, \Cref{thm:VanishingIdeal} shows that the vanishing ideal of $\CC M$ equals the conditional independence ideal $\ideal_{\BB R[\Sigma]}{|\Sigma_{ij|V\setminus ij}| : ij \not\in E}$ up to saturation at~$|\Sigma|$.
\end{example}

\pagebreak %

\begin{example}[RCON models] \label{ex:RCON}
Another related instance of linear concentration models are the \emph{restricted concentration models}, also known as \emph{colored undirected Gaussian graphical models}. Let~$G = (V, E)$ be an undirected graph, and $c\colon V \sqcup E \to C$ be a map which assigns a ``color'' from a finite set~$C$ to each vertex and edge, such that $c|_V$ and $c|_E$ have disjoint images.
The~RCON model for $(G, c)$ is then the set of covariance matrices $\Sigma$ reciprocal to the linear space of concentration matrices $K$ satisfying the equations on the left:
\begin{equation}
\label{eq:RCONimpl}
\begin{aligned}
  \kappa_{ij} &= 0, \text{if $ij \not\in E$} &&\quad\to\quad& |\Sigma_{ij|V\setminus ij}| &= 0, \\
  \kappa_{ii} &= \kappa_{jj}, \text{if $c(i) = c(j)$} &&\quad\to\quad& |\Sigma_{V \setminus i}| &= |\Sigma_{V \setminus j}|, \\
  \kappa_{ij} &= \kappa_{kl}, \text{if $c(ij) = c(kl)$} &&\quad\to\quad& |\Sigma_{ij|V \setminus ij}| &= |\Sigma_{kl|V \setminus kl}|.
\end{aligned}
\end{equation}
As in \Cref{ex:Undirected}, \Cref{thm:Biri} shows that the model is cut out of $\PD_n$ by the corresponding equations on the right, which are the numerators of the equations on the left under matrix inversion. Moreover, these polynomials generate the vanishing ideal up to saturation at~$|\Sigma|$.

\noindent\begin{sidebyside}
\begin{sidebox}{0.75\linewidth}
For a concrete example, consider the colored graph to the right (in~which all vertices have distinct~colors).
Collecting the equations from \eqref{eq:RCONimpl} corresponding to the missing edges and coloring conditions produces a (non-radical) ideal $\SR J \subseteq \BB R[\Sigma]$ of dimension~12 with 6 associated primes. Among them is a unique one which does not contain~$|\Sigma|$, namely
\end{sidebox}\hfill
\begin{sidebox}{0.25\linewidth}
\centering%
\noindent%
\begin{tikzpicture}[thick, scale=0.6]
  \node[circle, draw, inner sep=1pt, minimum width=1pt] (1) at (0,0)  {$1$};
  \node[circle, draw, inner sep=1pt, minimum width=1pt] (2) at (2,0)  {$2$};
  \node[circle, draw, inner sep=1pt, minimum width=1pt] (3) at (2,-2) {$3$};
  \node[circle, draw, inner sep=1pt, minimum width=1pt] (4) at (0,-2) {$4$};
  \node[circle, draw, inner sep=1pt, minimum width=1pt] (5) at (4,-1)  {$5$};

  \draw[red!40, ultra thick]  (1) -- (2);
  \draw[red!40, ultra thick]  (2) -- (3);
  \draw[red!40, ultra thick]  (3) -- (4);
  \draw[red!40, ultra thick]  (4) -- (1);
  \draw[blue!40, ultra thick] (2) -- (5);
  \draw[blue!40, ultra thick] (3) -- (5);
\end{tikzpicture}%
\end{sidebox}

\begingroup
\scriptsize
\begin{gather*}
  \big\langle\,
\sigma_{14} \sigma_{23} - \sigma_{12} \sigma_{34},\,
\sigma_{15} \sigma_{24} - \sigma_{14} \sigma_{25},\,
\sigma_{15} \sigma_{34} - \sigma_{14} \sigma_{35},\,
\sigma_{15} \sigma_{23} - \sigma_{12} \sigma_{35},\,
\sigma_{15} \sigma_{44} - \sigma_{14} \sigma_{45},\,
\sigma_{25} \sigma_{44} - \sigma_{24} \sigma_{45},\, \\
\sigma_{25} \sigma_{34} - \sigma_{24} \sigma_{35},\,
\sigma_{35} \sigma_{44} - \sigma_{34} \sigma_{45},\,
\sigma_{25} \sigma_{33} + \sigma_{14} \sigma_{35} - \sigma_{23} \sigma_{35} - \sigma_{13} \sigma_{45},\,
\sigma_{24} \sigma_{33} + \sigma_{14} \sigma_{34} - \sigma_{23} \sigma_{34} - \sigma_{13} \sigma_{44},\, \\
\sigma_{23} \sigma_{25} - \sigma_{22} \sigma_{35} - \sigma_{24} \sigma_{35} + \sigma_{23} \sigma_{45},\,
\sigma_{13} \sigma_{25} - \sigma_{12} \sigma_{35} - \sigma_{14} \sigma_{35} + \sigma_{13} \sigma_{45},\,
\sigma_{23} \sigma_{24} - \sigma_{22} \sigma_{34} - \sigma_{24} \sigma_{34} + \sigma_{23} \sigma_{44},\, \\
\sigma_{13} \sigma_{24} - \sigma_{12} \sigma_{34} - \sigma_{14} \sigma_{34} + \sigma_{13} \sigma_{44},\,
\sigma_{15} \sigma_{22} - \sigma_{12} \sigma_{25} + \sigma_{14} \sigma_{25} - \sigma_{12} \sigma_{45},\,
\sigma_{14} \sigma_{22} - \sigma_{12} \sigma_{24} + \sigma_{14} \sigma_{24} - \sigma_{12} \sigma_{44},\, \\
\sigma_{13} \sigma_{22} - \sigma_{12} \sigma_{23} + \sigma_{14} \sigma_{34} - \sigma_{13} \sigma_{44},\,
\sigma_{13} \sigma_{15} + \sigma_{15} \sigma_{33} - \sigma_{11} \sigma_{35} - \sigma_{13} \sigma_{35},\,
\sigma_{12} \sigma_{15} - \sigma_{11} \sigma_{25} - \sigma_{14} \sigma_{35} + \sigma_{13} \sigma_{45},\, \\
\sigma_{13} \sigma_{14} + \sigma_{14} \sigma_{33} - \sigma_{11} \sigma_{34} - \sigma_{13} \sigma_{34},\,
\sigma_{12} \sigma_{14} - \sigma_{11} \sigma_{24} - \sigma_{14} \sigma_{34} + \sigma_{13} \sigma_{44},\,
\sigma_{12} \sigma_{13} - \sigma_{11} \sigma_{23} - \sigma_{13} \sigma_{23} + \sigma_{12} \sigma_{33}
  \,\big\rangle.
\end{gather*}
\endgroup
\end{sidebyside}

\noindent%
This prime ideal of the expected dimension 7 is the vanishing ideal of the RCON~model.
\end{example}

\begin{figure}
\begin{subfigure}[t]{0.45\linewidth}
\centering
\includegraphics[width=0.85\linewidth]{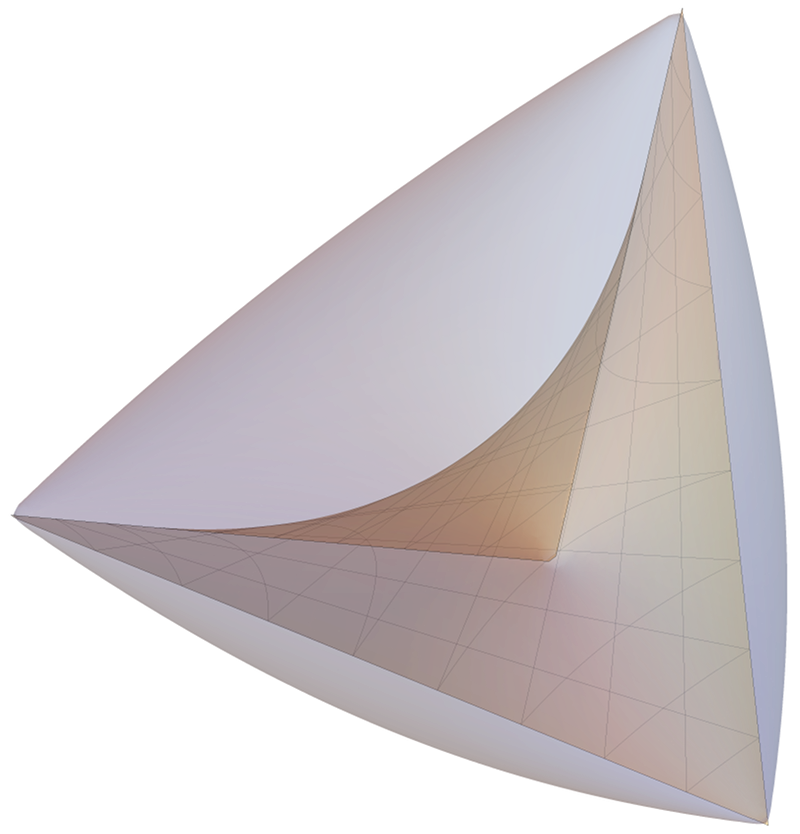}
\caption{The Gaussian conditional independence model for~$\CI{1,2|3}$ inside of the elliptope.}
\label{fig:MTP2:a}
\end{subfigure}
\hfill
\begin{subfigure}[t]{0.45\linewidth}
\centering
\includegraphics[width=0.85\linewidth]{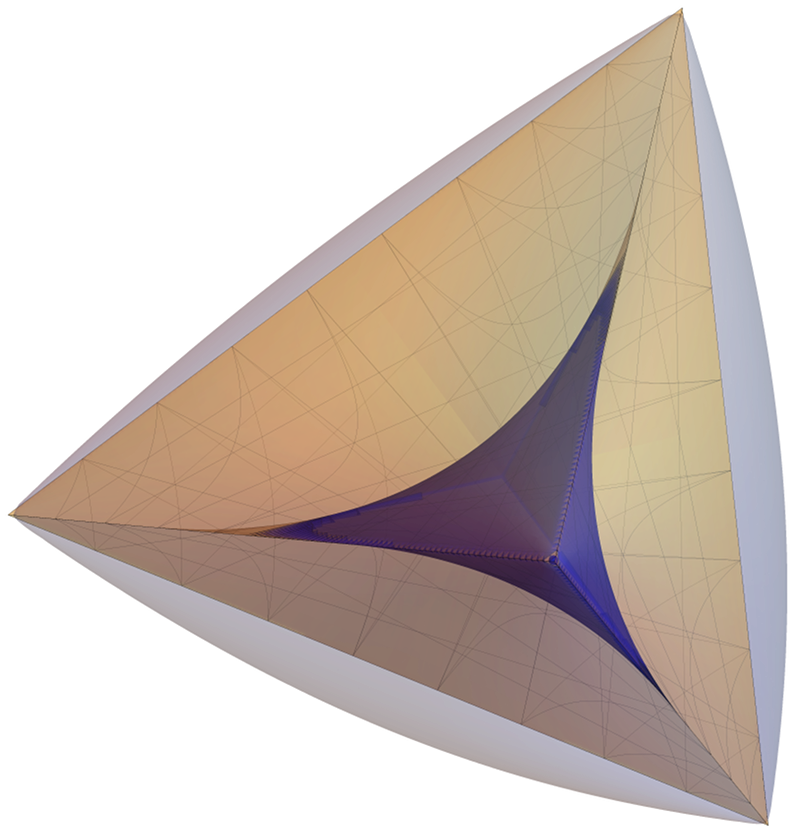}
\caption{The $\RM{MTP}_2$ distributions with normed variances form the dark arrow-shaped region.}
\label{fig:MTP2:b}
\end{subfigure}
\hfill
\caption{A linear and a polyhedral concentration model.}
\end{figure}

\begin{example}[Polyhedral constraints and $\RM{MTP}_2$] \label{ex:MTP2}
Our techniques transparently accommodate inequalities constraining the parameter space (in addition to the positivity of denominators). A natural step up from linear concentration models is to consider \emph{polyhedral concentration models} featuring linear inequalities. For example, given a linear space $\CC L$ in concentration coordinates, consider the polyhedral cone $\CC P = \Set{ K \in \CC L : \kappa_{ij} \le 0 }$. The parameter space $\CC P \cap \PD_n$ consists of the $\RM{M}$-matrices from $\CC L \cap \PD_n$. Correspondingly, the polyhedral concentration model $\CC M(\CC P) = \Set{ \Sigma \in \PD_n : \Sigma^{-1} \in \CC P \cap \PD_n }$ consists of the $\RM{MTP}_2$ distributions in~$\CC M(\CC L)$.

The covariance matrices of 3-variate Gaussians with normalized variance form a 3-dimensional convex body known as the \emph{elliptope}. The three undirected graphical models below are given by the parameter constraints $\kappa_{12} = 0$, $\kappa_{13} = 0$ and $\kappa_{23} = 0$, respectively, which map to the non-linear constraints $\sigma_{12} = \sigma_{13}\sigma_{23}$, $\sigma_{13} = \sigma_{12}\sigma_{23}$ and $\sigma_{23} = \sigma_{12}\sigma_{13}$, respectively, in the model space. The~three hypersurfaces only differ by a permutation of the ambient space; one of them is pictured in \Cref{fig:MTP2:a}.
\begin{center}
\begin{tikzpicture}[thick, scale=0.5]
  \node[circle, draw, inner sep=1pt, minimum width=1pt] (1) at (0,0)  {$1$};
  \node[circle, draw, inner sep=1pt, minimum width=1pt] (2) at (0,-2) {$2$};
  \node[circle, draw, inner sep=1pt, minimum width=1pt] (3) at (2,-1) {$3$};

  \draw[ultra thick] (1) -- (3);
  \draw[ultra thick] (2) -- (3);
\end{tikzpicture}
\hspace{2em}
\begin{tikzpicture}[thick, scale=0.5]
  \node[circle, draw, inner sep=1pt, minimum width=1pt] (1) at (0,0)  {$1$};
  \node[circle, draw, inner sep=1pt, minimum width=1pt] (2) at (0,-2) {$2$};
  \node[circle, draw, inner sep=1pt, minimum width=1pt] (3) at (2,-1) {$3$};

  \draw[ultra thick] (1) -- (2);
  \draw[ultra thick] (2) -- (3);
\end{tikzpicture}
\hspace{2em}
\begin{tikzpicture}[thick, scale=0.5]
  \node[circle, draw, inner sep=1pt, minimum width=1pt] (1) at (0,0)  {$1$};
  \node[circle, draw, inner sep=1pt, minimum width=1pt] (2) at (0,-2) {$2$};
  \node[circle, draw, inner sep=1pt, minimum width=1pt] (3) at (2,-1) {$3$};

  \draw[ultra thick] (1) -- (2);
  \draw[ultra thick] (1) -- (3);
\end{tikzpicture}
\end{center}
The inequality $\kappa_{ij} \le 0$ similarly maps to $(-1)^{i+j}\sigma_{ij} \le (-1)^{i+j}\sigma_{ik}\sigma_{jk}$. The set of $\RM{MTP}_2$ distributions inside of the elliptope is thus found on the negative side of each of the three hypersurfaces. It forms an arrow-shaped cell in the hypersurface arrangement depicted in \Cref{fig:MTP2:b}.
\end{example}

\subsection{Linear structural equation models}

Let $D_n$ be the space of $n \times n$ diagonal matrices and $T_n$ the space of $n \times n$ strictly upper-triangular matrices. For given $\Omega \in D_n \cap \PD_n$ and $\Lambda \in T_n$, a~random vector $X$ is defined as the unique solution to the \emph{linear structural equations}
\begin{equation}
  \label{eq:SEM}
  X = \Lambda X + \eps,
\end{equation}
where $\eps \sim \CC N(0, \Omega)$ is a vector of mutually independent normally distributed noise. It can be worked out (see, e.g., \cite[Proposition~13.2.12]{Sullivant}) that $X$ is itself normally distributed with mean $0$ and covariance matrix $\Sigma = \alpha(\Omega, \Lambda) \defas (I_n - \Lambda)^{-\T} \,\Omega\, (I_n - \Lambda)^{-1}$. Since $|I_n-\Lambda| = 1$, this is a polynomial~map $\alpha\colon D_n \times T_n \to \Sym_n$. By \cite[Lemma~3.2]{ColoredDAG} applied to the complete topologically ordered DAG on $V = \Set{1, \dots, n}$, its inverse is given by
\begin{align}
  \label{eq:DAGrecover}
  \omega_{ii}  = \frac{|\Sigma_{[i]}|}{|\Sigma_{[i-1]}|}, \quad
  \lambda_{ij} = \frac{|\Sigma_{ij|[j-1]\setminus i}|}{|\Sigma_{[j-1]}|}, \;\text{for $i < j$}.
\end{align}
This~gives rise to a birational isomorphism $S^{-1} \BB R[\Sigma] \iso \BB R[\Omega, \Lambda]$ where $S$ is the monoid generated by the leading principal minors of~$\Sigma$ which appear as denominators in \eqref{eq:DAGrecover}. The other ring has no localizing monoid attached because the map $\alpha$ is polynomial, so clearly this isomorphism is not full. Localizing at $\monoid{\omega_{ii} : i \in V}$ on the one side and at $\monoid{|\Sigma_{[i]}| : i \in V}$ on the other extends this to a full birational isomorphism.
Note that \Cref{thm:Biri} only applies to parameter spaces inside $D_n \times T_n$ on which the localizing monoid is strictly positive, i.e., $\omega_{ii} > 0$ for all $i \in V$. This,~in turn, places the model into~$\PD_n$, which is consistent with the statistical setup in~\eqref{eq:SEM}.

\begin{example}[Bayesian networks] \label{ex:Bayesian}
Let $G = (V,E)$ be a directed acyclic graph (DAG) on $|V|=n$ vertices which are topologically ordered. This defines a linear subspace $\CC L(E) \subseteq T_n$ by setting $\lambda_{ij} = 0$ for $ij \not\in E$. The topological ordering of $G$ ensures that all non-zero coefficients $\lambda_{ij}$, for $ij \in E$, are in the upper triangle of~$\Lambda$. This defines the parameter space $\CC S(G) = (D_n \cap \PD_n) \times \CC L(E)$ and the resulting \emph{Gaussian Bayesian network model}~$\CC M(G) = \alpha(\CC S(G))$. Combining \Cref{thm:Biri} with the formula for $\lambda_{ij}$ in \eqref{eq:DAGrecover} shows that this model is cut out of $\PD_n$ by the equations
\[
  |\Sigma_{ij|[j-1]\setminus i}| = 0, \text{for $i < j$ and $ij \not\in E$}.
\]
These vanishings correspond to the conditional independence statements $\CI{i,j|[j-1]\setminus i}$ which therefore characterize the model geometrically. \Cref{thm:VanishingIdeal} shows, moreover, that the vanishing ideal of the model is generated by these conditional independence polynomials up to saturation at the leading principal minors of~$\Sigma$.

Markov properties of DAG models are well-understood. Using the intersection and composition properties of Gaussian conditional independence and ideas from \cite{MatusGaussian}, one can prove that the vanishing ideal is equivalently generated by the polynomials corresponding to the \emph{pairwise directed local Markov property} $\CI{i,j|\pa(j)}$, for $ij \not\in E$, up to saturation at the principal minors of~$\Sigma$; cf.~\cite[Definition~13.1.9]{Sullivant}.
This proves a conjecture of Sullivant \cite{SullivantGaussianBN} which had been resolved in \cite{RoozbehaniPolyanskiy}. With more work, as carried out in \cite[Theorem~4.11]{ColoredDAG}, the saturating set can be reduced to only \emph{parental} principal minors $|\Sigma_{\pa(j)}|$, for $j \in V$.
\end{example}

\begin{example}[Confounding]
The influence of latent confounders in a Bayesian network is modeled using mixed graphs $G = (V,D,B)$ where $V = [n]$ is the vertex set, $D$ is a set of directed edges and $B$ a set of bidirected edges. A bidirected edge is drawn between vertices which have a common confounder; see~\cite[Section~14.2]{Sullivant} and \cite{AlgebraicProblems}. The parametrization $\Sigma = (I-\Lambda)^{-\T} \, \Omega \, (I-\Lambda)^{-1}$ is analogous to Bayesian networks, however, $\Omega$ need not be diagonal. Each bidirected edge $i\leftrightarrow j \in B$ couples the error terms $\eps_i$ and $\eps_j$ in the structural equation model~\eqref{eq:SEM} by allowing $\omega_{ij}$ to be non-zero.

With $\binom{n+1}{2} + \binom{n}{2}$ parameters mapping into an $\binom{n+1}{2}$-dimensional space, the general mixed graph model (with a complete DAG and all bidirected edges) cannot be identifiable and thus not ambirational; one has to trade directed edges for bidirected ones. A characterization of which graphical structures enjoy global rational identifiability is given in \cite[Theorem~16.2.1]{Sullivant}. We~examine only a particular case in which $2 \rightarrow 4$ is replaced by the bidirected edge~$2 \leftrightarrow 4$:
\begin{center}
\begin{tikzpicture}[thick, scale=0.9]
  \tikzset{>={Stealth[length=3mm, round]}}
  \node[circle, draw, inner sep=1pt, minimum width=1pt] (1) at (0,0) {$1$};
  \node[circle, draw, inner sep=1pt, minimum width=1pt] (2) at (2,0) {$2$};
  \node[circle, draw, inner sep=1pt, minimum width=1pt] (3) at (4,0) {$3$};
  \node[circle, draw, inner sep=1pt, minimum width=1pt] (4) at (6,0) {$4$};

  \draw (1) edge [->, very thick] (2);
  \draw (1) edge [->, very thick, bend right=30] (3);
  \draw (1) edge [->, very thick, bend right=40] (4);
  \draw (2) edge [->, very thick] (3);
  \draw (2) edge [<->, red!60, very thick, bend left=50] (4);
  \draw (3) edge [->, very thick] (4);
\end{tikzpicture}
\end{center}
This graphical model is globally identifiable:
\begin{gather*}
\omega_{11} = |\Sigma_{1}|, \;
\omega_{22} = \frac{|\Sigma_{12}|}{|\Sigma_{1}|}, \;
\omega_{33} = \frac{|\Sigma_{123}|}{|\Sigma_{12}|}, \;
\omega_{44} = \frac{|\Sigma_{1234}| }{|\Sigma_{123}|} + \frac{|\Sigma_{12}| |\Sigma_{24|13}|^2}{|\Sigma_{1}| |\Sigma_{123}|^2}, \;
\omega_{24} = \frac{|\Sigma_{12}| |\Sigma_{24|13}|}{|\Sigma_{1}| |\Sigma_{123}|}, \\
\lambda_{12} = \frac{|\Sigma_{12|\emptyset}|}{|\Sigma_{1}|}, \;
\lambda_{13} = \frac{|\Sigma_{13|2}|}{|\Sigma_{12}|}, \;
\lambda_{14} = \frac{|\Sigma_{14|23}|}{|\Sigma_{123}|} + \frac{|\Sigma_{12|\emptyset}| |\Sigma_{24|13}|}{|\Sigma_{1}| |\Sigma_{123}|}, \;
\lambda_{23} = \frac{|\Sigma_{23|1}|}{|\Sigma_{12}|}, \;
\lambda_{34} = \frac{|\Sigma_{34|12}|}{|\Sigma_{123}|}.
\end{gather*}
The positive definiteness of $\Omega$ imposes 4 inequality conditions on the image:
\begin{equation*}
\begin{aligned}
  \omega_{11} &> 0 &&\quad\to\quad& |\Sigma_{1}| &> 0, \\
  \omega_{11} \omega_{22} &> 0 &&\quad\to\quad& |\Sigma_{12}| &> 0, \\
  \omega_{11}\omega_{22}\omega_{33} &> 0 &&\quad\to\quad& |\Sigma_{123}| &> 0, \\
  \omega_{11}\omega_{22}\omega_{33}\omega_{44} - \omega_{11}\omega_{33}\omega_{24}^2 &> 0 &&\quad\to\quad& |\Sigma_{1234}| &> 0.
\end{aligned}
\end{equation*}
Hence, when localizing at the leading principal minors of $\Omega$ on the one side and at those of~$\Sigma$ on the other side, this mixed graphical model induces an affine birational isomorphism, allowing us to study its submodels. Removing the edge $1 \rightarrow 4$ results in the well-known \emph{Verma graph} \cite[Example~1.2]{Nested}. Its model is cut out of $\PD_4$ by the \emph{Verma constraint} which is not of conditional independence type. In \cite{Nested}, this constraint is derived using nested determinants. Using birational implicitization, we can explain its occurence as the numerator of the image of $\lambda_{14} = 0$ under the parameter identification map:
$\lambda_{14} = 0 \;\to\; |\Sigma_{1}| |\Sigma_{14|23}| + |\Sigma_{12|\emptyset}| |\Sigma_{24|13}| = 0$.
\end{example}

\begin{remark}
Latent variable models, in general, are overparametrized and non-identifiable and our techniques do not apply. For these cases, recent work by Cummings and Hollering \cite{Benjoe} exploits a multigrading compatible with the parametrization map $\alpha$ to break the implicitization problem into smaller subproblems which can be solved via linear algebra. In contrast to general elimination-theoretic methods, their multigraded implicitization can quickly find generating sets for low-degree pieces of the vanishing ideal which has led to impressive speedups in model distinguishability benchmarks in algebraic~phylogenetics.
\end{remark}

\subsection{Interventions and colored DAGs}

One of the main goals in graphical modeling is to infer which graph (from a given class) is most likely to produce a given data sample. For~Bayesian networks, this task is obstructed by the phenomenon of \emph{model equivalence}, namely that two distinct DAGs $G \not= H$ may define the same model $\CC M(G) = \CC M(H)$. The field of \emph{causality} deals with the problem of selecting the right DAG after its model equivalence class has been determined~\cite{pearl2009causality}.
Two solutions to this structural identifiability problem have resulted in generalizations of DAG models that share in the algebraic structure of~$\CC M(G)$: \emph{colored Gaussian DAG models} \cite{peters2014identifiability,wu2023partial,ColoredDAG} and \emph{interventional DAG models} \cite{yang2018characterizing}.

\begin{example}[Colored DAG models] \label{ex:ColoredDAG}
Colored DAG models are a natural extension of Bayesian networks parallel to \Cref{ex:RCON}. Let $G = (V,E)$ be a DAG, for simplicity $V = [n]$ labeled according to a topological ordering of the DAG, and $c: V \sqcup E \to C$ a coloring function with the images of $c|_V$ and $c|_E$ disjoint. The coloring introduces additional equations $\omega_{ii} = \omega_{jj}$ for $c(i) = c(j)$ and $\lambda_{ij} = \lambda_{kl}$ for $c(ij) = c(kl)$. It is known that vertex-colored DAGs and edge-colored DAGs (i.e., those where $c|_E$, respectively $c|_V$, is injective) refine their underlying DAG's model equivalence class. This is open in general for mixed colorings; see~\cite{wu2023partial} and \cite[Section~5]{ColoredDAG}.

Using \Cref{thm:VanishingIdeal} and \eqref{eq:DAGrecover} it is again easy to write down a generating set for the vanishing ideal of the resulting model $\CC M(G,c)$ up to saturation at the leading principal minors:
\begin{align}
  \label{eq:DAGCI}  |\Sigma_{ij|[j-1]}| &= 0, \; \text{for $i < j$ and $ij \not\in E$}, \\
  \label{eq:Vcolor} |\Sigma_{[i]}| \, |\Sigma_{[j-1]}| &= |\Sigma_{[j]}| \, |\Sigma_{[i-1]}|, \; \text{if $c(i) = c(j)$}, \\
  \label{eq:Ecolor} |\Sigma_{ij|[j-1]\setminus i}| \, |\Sigma_{[l-1]}| &= |\Sigma_{kl|[l-1]\setminus k}| \, |\Sigma_{[j-1]}|, \; \text{if $c(ij) = c(kl)$}.
\end{align}

\begin{figure}[t]
\begin{subfigure}[t]{0.43\linewidth}
\centering
\includegraphics[width=0.85\linewidth]{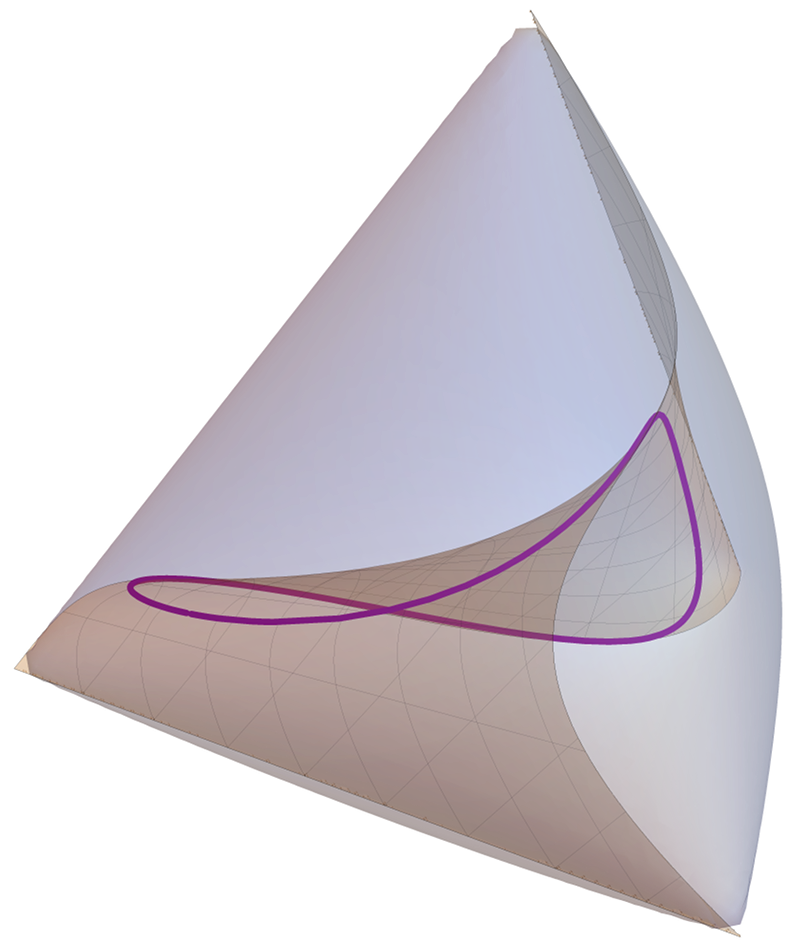}
\caption{The~distributions satisfying the ball constraint~\eqref{eq:Ball} lie inside of the purple circle.}
\label{fig:ColoredBall}
\end{subfigure}
\hfill
\begin{subfigure}[t]{0.50\linewidth}
\centering
\includegraphics[width=0.9\linewidth]{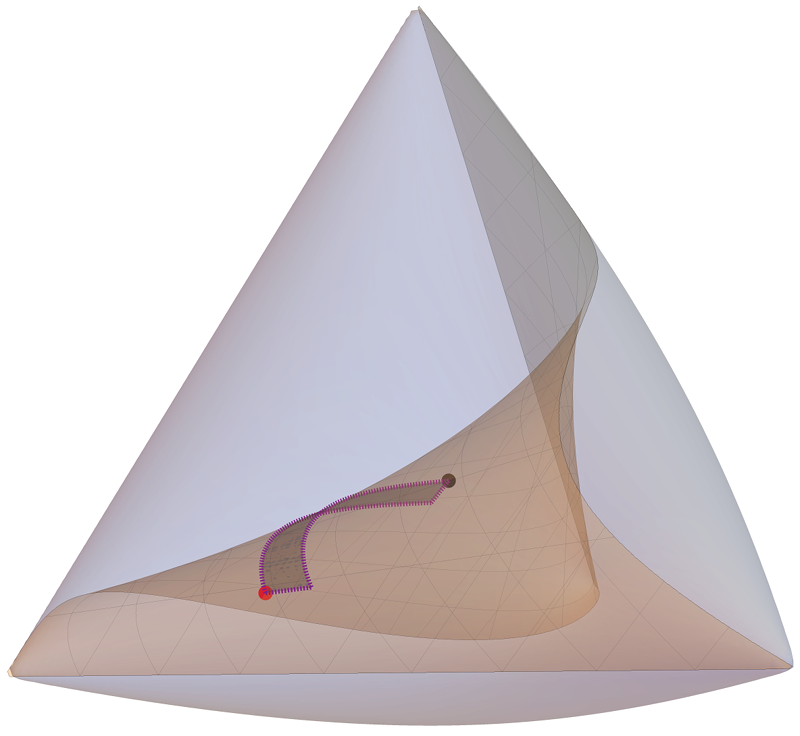}
\caption{A successful intervention on the red distribution which decreases both structural coefficients in absolute value lands in the shaded strip. The black dot at the opposite end is the standard normal~$\CC N(0,I_3)$.}
\label{fig:Intervention}
\end{subfigure}
\caption{Semialgebraic submodels of the colored DAG model from \Cref{ex:ColoredDAG}.} %
\end{figure}

\noindent\begin{sidebyside}
\begin{sidebox}{0.85\linewidth}
Let $(G, c)$ be the colored DAG to the right and consider the submodel of $\CC M(G, c)$ given by the additional ball constraint $2\lambda_{12}^2 + \lambda_{23}^2 < 1$. Since $\lambda_{12} = \lambda_{13}$, this restricts the structural (i.e., edge) coefficients to the open unit ball.
We obtain the following implicit description (omitting $\PD$ constraints):
\begin{gather}
  \label{eq:Color} \sigma_{12} (1-\sigma_{12}^2) = \sigma_{13} - \sigma_{12} \sigma_{23}, \\
  \label{eq:Ball}  \sigma_{12}^2 + 2 (\sigma_{13} - \sigma_{12} \sigma_{23})^2 + (\sigma_{23} - \sigma_{12}\sigma_{13})^2 < 1.
\end{gather}
\vspace{-0.5\baselineskip}
\end{sidebox}\hfill
\begin{sidebox}[t]{0.15\linewidth}
\centering%
\noindent%
\begin{tikzpicture}[thick, scale=0.6]
  \tikzset{>={Stealth[length=3mm, round]}}
  \node[circle, draw, inner sep=1pt, minimum width=1pt] (1) at (0,0)  {$1$};
  \node[circle, draw, inner sep=1pt, minimum width=1pt] (2) at (0,-2) {$2$};
  \node[circle, draw, inner sep=1pt, minimum width=1pt] (3) at (2,-1) {$3$};

  \draw[->, very thick, cyan] (1) -- (2);
  \draw[->, very thick, cyan] (1) -- (3);
  \draw[->, very thick]       (2) -- (3);
\end{tikzpicture}
\end{sidebox}
The colored DAG model and its submodel are depicted in~\Cref{fig:ColoredBall}. \qedhere
\end{sidebyside}
\end{example}

The classic approach to distinguishing a single DAG from within its model equivalence class is to use a mixture of \emph{observational} data and \emph{interventional} (or \emph{experimental}) data; for instance, the two data sets produced from a randomized controlled trial. In interventional models, we assume that there is a data-generating distribution~$X$, from which the observational data is drawn.

Consider a DAG $G = (V,E)$ with its parameter space $\CC S(G)$ as defined in \Cref{ex:Bayesian}, and fix a set $I \subseteq V$ of \emph{intervention targets}. Let $\Sigma = \alpha(\Omega, \Lambda)$, with $(\Omega, \Lambda) \in \CC S(G)$, be the observational distribution. The~\emph{interventional distribution} is produced by perturbing (``experimenting on'') the conditional factors ${X_i \mid X_{\pa(i)}}$, for $i \in I$, in some way. For Gaussian DAG models, this may be modeled by setting $\Sigma' = \alpha(\Omega', \Lambda')$ with additional parameters $(\Omega', \Lambda') \in \CC S(G)$. To enforce that only~$I$ is being intervened on, these parameters are subject to $\lambda_{ji}' = \lambda_{ji}$ and $\omega_{ii}' = \omega_{ii}$ whenever~$i \not\in I$. Hence, the effects of the perturbation at the nodes $I$ is captured in the difference between the observational parameters and the interventional parameters. Note that the choice of intervention targets $I = \emptyset$ returns the observational DAG~model.

If the parametrization maps $\alpha_0$ and $\alpha_1$ induce affine birational isomorphisms, then so does their direct product $\alpha_0 \times \alpha_1$ on the cartesian product of their domains and codomains.
From a sequence of intervention targets $\CC I = (I_0 = \emptyset, I_1, \ldots, I_m)$, we obtain a parameter space $\CC S(G)^m / {\sim}$ where the relation ${\sim}$ imposes the equalities among non-targeted parameters between successive interventions. This parameter space is again a linear space $\CC L(\CC I)$ in $\PD_n^m \times T_n^m$. The resulting \emph{interventional Gaussian DAG model} $\CC M(G, \CC I)$ is the image of $\CC L(\CC I)$ under an $m$-fold direct product of the parametrization maps~$\alpha$ for standard Bayesian networks.

\begin{remark}
Given a sequence of intervention targets $\CC I$, the $\CC I$-model equivalence classes refine the DAG model equivalence classes. When $\CC I = (I_\emptyset, (I_i = \{i\} : i \in V))$, each $\CC I$-model equivalence class is a singleton, which solves the structural identifiability problem when such interventional data is available. Although in practice performing such experiments may be costly or unethical.
\end{remark}

\begin{remark}
Colored Gaussian DAG models generalize interventional DAG models with intervention targets $\CC I = (I_0, \dots, I_m)$ by taking the graph $H$ as $m+1$ disjoint copies of $G$ and the coloring $c$ that colors the vertices and edges of $H$ according to the linear space $\CC L(\CC I)$ defined~above.
\end{remark}

\begin{example}[Error variance interventions] \label{ex:ErrorIntervention}
A popular type of interventional model is to assume that an intervention on nodes $I \subseteq V$ affects the conditional factors $X_i \mid X_{\pa(i)}$, for $i \in I$, only through their error variances~$\omega_{ii}$.
Fix a DAG $G = (V,E)$, an observational distribution~$X \sim \CC N(0, \Sigma)$, for $\Sigma \in \CC M(G)$, and intervention targets~$I$. Let $X' \sim \CC N(0, \Sigma')$ be the interventional distribution. Recasting the intervention as a colored DAG model and the resulting relations \labelcref{eq:DAGCI,eq:Vcolor,eq:Ecolor} on (the covariance matrix of) the pair $(X, X')$ in statistical terminology, the local Markov property of the DAG implies
\begin{align}
  \label{eq:DAGCIStat} \tag{$\text{\ref*{eq:DAGCI}}'$} \CI{X_i,X_j|X_{\pa(j)}} \,\,\,&\!\!\!\text{ and } \CI{X_i',X_j'|X_{\pa(j)}'}, \; \text{for $i < j$ and $ij \not\in E$}, \\
  \label{eq:VcolorStat} \tag{$\text{\ref*{eq:Vcolor}}'$} \Var[X_i \mid X_{\pa(i)}] &= \Var[X_i' \mid X_{\pa(i)}'], \; \text{for $i \not\in I$}, \\
  \label{eq:EcolorStat} \tag{$\text{\ref*{eq:Ecolor}}'$} \frac{\Cov[X_i,X_j \mid X_{\pa(j)\setminus i}]}{\Var[X_i \mid X_{\pa(j)\setminus i}]} &= \frac{\Cov[X_i',X_j' \mid X_{\pa(j)\setminus i}']}{\Var[X_i' \mid X_{\pa(j)\setminus i}']}, \; \text{for $ij \in E$}.
\end{align}
In terms of conditional densities, this gives
\[
f(x_i \mid x_{\pa_G(i)}) = f^\prime(x_i \mid x_{\pa_G(i)}) \quad \mbox{for all $i \notin I$}
\]
where $f(x)$ and $f^\prime(x)$ denote the joint densities of $X$ and $X^\prime$, respectively. These are the recognizable \emph{causal mechanism invariance constraints} commonly used in causality to define interventional models \cite{tian2001causal}.
\end{example}

\Cref{thm:Biri} also allows us to relax the classic interventional model studied in the previous example. This allows practitioners to build expert knowledge into the specification of the parameter space and the intervention and still recover constraints for testing model membership against data.

\begin{example}[Monotone interventions]
\label{ex:MonotoneIntervention}
When an intervention is performed, it may often be assumed that the effect of the perturbation is to either increase or decrease entropy in the targeted causal mechanisms. For instance, a perfect intervention (see \cite{hauser2015jointly} for definition) in $X_i \mid X_{\pa(i)}$ may decrease the error variance to zero and specify a certain outcome of this conditional factor. Alternatively, an intervention may lead to unexpected outcomes, in which case a reasonable assumption is that the error variances of targeted nodes should increase from their baseline. In the context of the model considered in Example~\ref{ex:ErrorIntervention}, the latter can be modeled by adding in the linear inequality constraints $\omega_{ii}' \geq \omega_{ii}$ for all $i\in I$. By \eqref{eq:DAGrecover} and \Cref{thm:Biri}, this corresponds to non-linear inequalities among products of principal minors: $|\Sigma_{[i]}|\,|\Sigma_{[i-1]}'| \geq |\Sigma_{[i-1]}|\,|\Sigma_{[i]}'|$. %

\Cref{fig:Intervention} shows an observational distribution (in red) on the colored DAG model from \Cref{ex:ColoredDAG}. The shaded area contains all points on the model which have lower structural coefficients $\lambda_{12}'$ and $\lambda_{23}'$ in absolute value. A successful intervention reducing the causal effects must generate a distribution in this strip. In the limit, the causal mechanisms are removed and the standard normal distribution (the black dot) is obtained.
\end{example}

\subsection{Staged tree models}
\label{subsec:StagedTree}

Staged trees parametrize statistical models via probability tree diagrams. As such they lend themselves well to the specification of constraints among conditional probabilities. Unlike all previous examples, which were multivariate Gaussians, staged tree models describe discrete random variables.

Let $T = (V, E)$ be a rooted tree in which every vertex $v \in V$ has either no or at least two outgoing edges. For each $v \in V$, the set of outgoing edges is denoted $E(v) \subseteq E$. To each edge $e\in E$, we assign a parameter $\theta_e$ assuming values in the interval $(0,1)$. At each vertex, this gives a vector of outgoing parameters $\theta_v = (\theta_e : e \in E(v))$ which we suppose sum to~$1$. Hence $\theta_v$ ranges in the open probability simplex $\Delta^\circ(E(v))$ of discrete distributions attaining $|E(v)|$ outcomes. The \emph{probability tree model} associated to~$T$ has parameter space $\Theta = \bigtimes_{v \in V} \Delta^\circ(E(v))$. Let $L(T)$ be the root-to-leaf paths in~$T$ and to each $\lambda \in L(T)$ associate the set $E(\lambda)$ of edges on~$\lambda$. Then the model is the image of the parametrization map
\begin{align*}
  \alpha_T \colon \Theta &\to \Delta^\circ(L(T)), \\
  \theta &\mathrel{\Arrow[|->]} \left(p_\lambda = \prod_{e \in E(\lambda)} \theta_e : \lambda \in L(T) \right).
\end{align*}
A random variable in this model $\CC M(T)$ may be simulated as follows: starting at the root node, pick one of the outgoing edges according to their probability~$\theta_e$ and continue until a leaf is reached. Usually, the levels of the tree are associated with random variables, so that a leaf corresponds to the outcome of a random vector; see \Cref{fig: equivalent staged trees}. The edge parameters represent conditional probabilities given the event that its ancestors were visited.

Given $p \in \CC M(T)$, the marginal probability of arriving at the event $v$ may be computed by summing over all coordinates $p_\lambda$ of $p$ where $\lambda$ is a root-to-leaf path passing through $v$:
\[
  p_{[v]} = \sum_{v \in \lambda \in L(T)} p_\lambda.
\]
As noted by Duarte and G\"orgen in \cite[Lemma~1]{duarte2020equations}, the edge parameter $\theta_{vw}$ is recovered as $\theta_{vw} = \sfrac{p_{[w]}}{p_{[v]}}$. Clearly, localizing $\BB R[\theta_e : e \in E]$ at $\<S = \monoid{\theta_e : e \in E}$ on the one side and $\BB R[p_\lambda : \lambda \in L(T)]$ at $S = \monoid{p_{[v]} : v \in V}$ on the other yields a full birational isomorphism. This shows that probability tree models are ambirational.

A \emph{staged tree} is a rooted tree $T$ together with coloring $c: V \to C$ of its vertices, such that $\theta_v = \theta_w$ whenever $c(v) = c(w)$; this necessitates that $|E(v)| = |E(w)|$. The fibers $c^{-1}(k)$, for $k \in c(V)$, are referred to as \emph{stages}. The staged tree model is a submodel of the probability tree model of $T$ defined by the linear constraints
\begin{align}
  \notag \theta_e &> 0, \; \text{for all $e \in E$}, \\
  \notag \sum_{e \in E(v)} \theta_e &= 1, \; \text{for all $v\in V$ non-leaf}, \\
  \label{eq:Staged} \theta_v &= \theta_w \; \text{whenever $c(v) = c(w)$}.
\end{align}
The first two conditions are inherited from the probability tree model and place the image into the probability simplex $\Delta^\circ(L(T))$ (cf.~\cite{duarte2020equations}), while the last condition gives the model-defining constraints
\begin{equation}
  \label{eqn: model invariants}
  p_{[v]}p_{[w']} = p_{[v']}p_{[w]} \; \text{whenever $c(v) = c(v')$ and $vw, v'w' \in E$}.
\end{equation}
Letting $\SR J$ be the ideal generated by these constraints, it follows directly from \Cref{thm:VanishingIdeal} that $\SR J : S$ is the vanishing ideal of the model.
This result was first observed by Duarte and Görgen in \cite{duarte2020equations}, where they call the constraints~\eqref{eqn: model invariants} the \emph{model invariants} of $(T,c)$. %

\begin{remark}
\label{rem: discrete DAGs}
Any discrete DAG model $\CC M(G)$ may be represented via a staged tree, in the sense that given $G$, there exists some $(T,c)$ such that $\CC M(G) = \CC M(T, c)$; see \cite[Proposition~2.2]{duarte2023new} for a construction. Combining this observation with the \Cref{ex:Bayesian}, we see that both discrete and Gaussian DAG models fall into the general family of ambirational statistical models.
\end{remark}

\begin{figure}[t]
    \begin{subfigure}[b]{0.3\textwidth}
    \centering
    \begin{tikzpicture}[thick,xscale=0.15,yscale=0.25]
        \tikzset{>={Stealth[length=3mm, round]}}
		\node[draw, fill=black!0, inner sep=2pt, rounded corners, minimum width=2pt] (w3) at (6,15)  {\scriptsize 1111};
		\node[draw, fill=black!0, inner sep=2pt, rounded corners, minimum width=2pt] (w4) at (6,13.5) {\scriptsize 1110};
		\node[draw, fill=black!0, inner sep=2pt, rounded corners, minimum width=2pt] (w5) at (6,12) {\scriptsize 1101};
		\node[draw, fill=black!0, inner sep=2pt, rounded corners, minimum width=2pt] (w6) at (6,10.5) {\scriptsize 1100};
		\node[draw, fill=black!0, inner sep=2pt, rounded corners, minimum width=2pt] (v3) at (6,9)  {\scriptsize 1011};
		\node[draw, fill=black!0, inner sep=2pt, rounded corners, minimum width=2pt] (v4) at (6,7.5) {\scriptsize 1010};
		\node[draw, fill=black!0, inner sep=2pt, rounded corners, minimum width=2pt] (v5) at (6,6) {\scriptsize 1001};
		\node[draw, fill=black!0, inner sep=2pt, rounded corners, minimum width=2pt] (v6) at (6,4.5) {\scriptsize 1000};
		\node[draw, fill=black!0, inner sep=2pt, rounded corners, minimum width=2pt] (w3i) at (6,3)  {\scriptsize 0111};
		\node[draw, fill=black!0, inner sep=2pt, rounded corners, minimum width=2pt] (w4i) at (6,1.5) {\scriptsize 0110};
		\node[draw, fill=black!0, inner sep=2pt, rounded corners, minimum width=2pt] (w5i) at (6,0) {\scriptsize 0101};
		\node[draw, fill=black!0, inner sep=2pt, rounded corners, minimum width=2pt] (w6i) at (6,-1.5) {\scriptsize 0100};
		\node[draw, fill=black!0, inner sep=2pt, rounded corners, minimum width=2pt] (v3i) at (6,-3)  {\scriptsize 0011};
		\node[draw, fill=black!0, inner sep=2pt, rounded corners, minimum width=2pt] (v4i) at (6,-4.5) {\scriptsize 0010};
		\node[draw, fill=black!0, inner sep=2pt, rounded corners, minimum width=2pt] (v5i) at (6,-6) {\scriptsize 0001};
		\node[draw, fill=black!0, inner sep=2pt, rounded corners, minimum width=2pt] (v6i) at (6,-7.5) {\scriptsize 0000};

		\node[draw, fill=blue!0, inner sep=2pt, rounded corners, minimum width=2pt] (w1) at (-2,14.25) {\scriptsize 111};
		\node[draw, fill=cyan!60, inner sep=2pt, rounded corners, minimum width=2pt] (w2) at (-2,11.25) {\scriptsize 110};
		\node[draw, fill=orange!0, inner sep=2pt, rounded corners, minimum width=2pt] (v1) at (-2,8.25) {\scriptsize 101};
		\node[draw, fill=cyan!60, inner sep=2pt, rounded corners, minimum width=2pt] (v2) at (-2,5.25) {\scriptsize 100};
		\node[draw, fill=orange!60, inner sep=2pt, rounded corners, minimum width=2pt] (w1i) at (-2,2.25) {\scriptsize 011};
		\node[draw, fill=cyan!60, inner sep=2pt, rounded corners, minimum width=2pt] (w2i) at (-2,-0.75) {\scriptsize 010};
		\node[draw, fill=orange!60, inner sep=2pt, rounded corners, minimum width=2pt] (v1i) at (-2,-3.75) {\scriptsize 001};
		\node[draw, fill=cyan!60, inner sep=2pt, rounded corners, minimum width=2pt] (v2i) at (-2,-6.75) {\scriptsize 000};

		\node[draw, fill=green!60, inner sep=2pt, rounded corners, minimum width=2pt] (w) at (-8,12.75) {\scriptsize 11};
		\node[draw, fill=green!60, inner sep=2pt, rounded corners, minimum width=2pt] (v) at (-8,6.75) {\scriptsize 10};
		\node[draw, fill=cyan!0, inner sep=2pt, rounded corners, minimum width=2pt] (wi) at (-8,0.75) {\scriptsize 01};
		\node[draw, fill=cyan!0, inner sep=2pt, rounded corners, minimum width=2pt] (vi) at (-8,-5.25) {\scriptsize 00};

		\node[draw, fill=yellow!0, inner sep=2pt, rounded corners, minimum width=2pt] (r) at (-14,9.75) {\scriptsize 1};
		\node[draw, fill=yellow!0, inner sep=2pt, rounded corners, minimum width=2pt] (ri) at (-14,-1.75) {\scriptsize 0};

		\node[draw, fill=black!0, inner sep=2pt, rounded corners, minimum width=2pt] (I) at (-20,3) {\scriptsize r};

		\draw[->]   (I) --    (r) ;
		\draw[->]   (I) --   (ri) ;

		\draw[->]   (r) --   (w) ;
		\draw[->]   (r) --   (v) ;

		\draw[->]   (w) --  (w1) ;
		\draw[->]   (w) --  (w2) ;

		\draw[->]   (w1) --   (w3) ;
		\draw[->]   (w1) --   (w4) ;
		\draw[->]   (w2) --  (w5) ;
		\draw[->]   (w2) --  (w6) ;

		\draw[->]   (v) --  (v1) ;
		\draw[->]   (v) --  (v2) ;

		\draw[->]   (v1) --  (v3) ;
		\draw[->]   (v1) --  (v4) ;
		\draw[->]   (v2) --  (v5) ;
		\draw[->]   (v2) --  (v6) ;

		\draw[->]   (ri) --   (wi) ;
		\draw[->]   (ri) -- (vi) ;

		\draw[->]   (wi) --  (w1i) ;
		\draw[->]   (wi) --  (w2i) ;

		\draw[->]   (w1i) --  (w3i) ;
		\draw[->]   (w1i) -- (w4i) ;
		\draw[->]   (w2i) --  (w5i) ;
		\draw[->]   (w2i) --  (w6i) ;

		\draw[->]   (vi) --  (v1i) ;
		\draw[->]   (vi) --  (v2i) ;

		\draw[->]   (v1i) --  (v3i) ;
		\draw[->]   (v1i) -- (v4i) ;
		\draw[->]   (v2i) -- (v5i) ;
		\draw[->]   (v2i) --  (v6i) ;

		\node at (-17.5,-9) {$X_2$} ;
		\node at (-11.5,-9) {$X_1$} ;
		\node at (-5,-9) {$X_3$} ;
		\node at (2,-9) {$X_4$} ;

	\end{tikzpicture}
	\caption{}\label{fig:cstreeA}
    \end{subfigure}
    \hfill
    \begin{subfigure}[b]{0.3\textwidth}
    \centering
    \begin{tikzpicture}[thick,xscale=0.15,yscale=0.25]
      \tikzset{>={Stealth[length=3mm, round]}}
		\node[draw, fill=black!0, inner sep=2pt, rounded corners, minimum width=2pt] (w3) at (6,15)  {\scriptsize 1111};
		\node[draw, fill=black!0, inner sep=2pt, rounded corners, minimum width=2pt] (w4) at (6,13.5) {\scriptsize 1110};
		\node[draw, fill=black!0, inner sep=2pt, rounded corners, minimum width=2pt] (w5) at (6,12) {\scriptsize 1101};
		\node[draw, fill=black!0, inner sep=2pt, rounded corners, minimum width=2pt] (w6) at (6,10.5) {\scriptsize 1100};
		\node[draw, fill=black!0, inner sep=2pt, rounded corners, minimum width=2pt] (v3) at (6,9)  {\scriptsize 1011};
		\node[draw, fill=black!0, inner sep=2pt, rounded corners, minimum width=2pt] (v4) at (6,7.5) {\scriptsize 1010};
		\node[draw, fill=black!0, inner sep=2pt, rounded corners, minimum width=2pt] (v5) at (6,6) {\scriptsize 1001};
		\node[draw, fill=black!0, inner sep=2pt, rounded corners, minimum width=2pt] (v6) at (6,4.5) {\scriptsize 1000};
		\node[draw, fill=black!0, inner sep=2pt, rounded corners, minimum width=2pt] (w3i) at (6,3)  {\scriptsize 0111};
		\node[draw, fill=black!0, inner sep=2pt, rounded corners, minimum width=2pt] (w4i) at (6,1.5) {\scriptsize 0110};
		\node[draw, fill=black!0, inner sep=2pt, rounded corners, minimum width=2pt] (w5i) at (6,0) {\scriptsize 0101};
		\node[draw, fill=black!0, inner sep=2pt, rounded corners, minimum width=2pt] (w6i) at (6,-1.5) {\scriptsize 0100};
		\node[draw, fill=black!0, inner sep=2pt, rounded corners, minimum width=2pt] (v3i) at (6,-3)  {\scriptsize 0011};
		\node[draw, fill=black!0, inner sep=2pt, rounded corners, minimum width=2pt] (v4i) at (6,-4.5) {\scriptsize 0010};
		\node[draw, fill=black!0, inner sep=2pt, rounded corners, minimum width=2pt] (v5i) at (6,-6) {\scriptsize 0001};
		\node[draw, fill=black!0, inner sep=2pt, rounded corners, minimum width=2pt] (v6i) at (6,-7.5) {\scriptsize 0000};

		\node[draw, fill=blue!0, inner sep=2pt, rounded corners, minimum width=2pt] (w1) at (-2,14.25) {\scriptsize 111};
		\node[draw, fill=blue!00, inner sep=2pt, rounded corners, minimum width=2pt] (w2) at (-2,11.25) {\scriptsize 110};
		\node[draw, fill=orange!60, inner sep=2pt, rounded corners, minimum width=2pt] (v1) at (-2,8.25) {\scriptsize 101};
		\node[draw, fill=orange!60, inner sep=2pt, rounded corners, minimum width=2pt] (v2) at (-2,5.25) {\scriptsize 100};
		\node[draw, fill=cyan!60, inner sep=2pt, rounded corners, minimum width=2pt] (w1i) at (-2,2.25) {\scriptsize 011};
		\node[draw, fill=cyan!60, inner sep=2pt, rounded corners, minimum width=2pt] (w2i) at (-2,-0.75) {\scriptsize 010};
		\node[draw, fill=cyan!60, inner sep=2pt, rounded corners, minimum width=2pt] (v1i) at (-2,-3.75) {\scriptsize 001};
		\node[draw, fill=cyan!60, inner sep=2pt, rounded corners, minimum width=2pt] (v2i) at (-2,-6.75) {\scriptsize 000};

		\node[draw, fill=green!60, inner sep=2pt, rounded corners, minimum width=2pt] (w) at (-8,12.75) {\scriptsize 11};
		\node[draw, fill=cyan!0, inner sep=2pt, rounded corners, minimum width=2pt] (v) at (-8,6.75) {\scriptsize 10};
		\node[draw, fill=green!60, inner sep=2pt, rounded corners, minimum width=2pt] (wi) at (-8,0.75) {\scriptsize 01};
		\node[draw, fill=cyan!0, inner sep=2pt, rounded corners, minimum width=2pt] (vi) at (-8,-5.25) {\scriptsize 00};

		\node[draw, fill=yellow!0, inner sep=2pt, rounded corners, minimum width=2pt] (r) at (-14,9.75) {\scriptsize 1};
		\node[draw, fill=yellow!0, inner sep=2pt, rounded corners, minimum width=2pt] (ri) at (-14,-1.75) {\scriptsize 0};

		\node[draw, fill=black!0, inner sep=2pt, rounded corners, minimum width=2pt] (I) at (-20,3) {\scriptsize r};

		\draw[->]   (I) --    (r) ;
		\draw[->]   (I) --   (ri) ;

		\draw[->]   (r) --   (w) ;
		\draw[->]   (r) --   (v) ;

		\draw[->]   (w) --  (w1) ;
		\draw[->]   (w) --  (w2) ;

		\draw[->]   (w1) --   (w3) ;
		\draw[->]   (w1) --   (w4) ;
		\draw[->]   (w2) --  (w5) ;
		\draw[->]   (w2) --  (w6) ;

		\draw[->]   (v) --  (v1) ;
		\draw[->]   (v) --  (v2) ;

		\draw[->]   (v1) --  (v3) ;
		\draw[->]   (v1) --  (v4) ;
		\draw[->]   (v2) --  (v5) ;
		\draw[->]   (v2) --  (v6) ;

		\draw[->]   (ri) --   (wi) ;
		\draw[->]   (ri) -- (vi) ;

		\draw[->]   (wi) --  (w1i) ;
		\draw[->]   (wi) --  (w2i) ;

		\draw[->]   (w1i) --  (w3i) ;
		\draw[->]   (w1i) -- (w4i) ;
		\draw[->]   (w2i) --  (w5i) ;
		\draw[->]   (w2i) --  (w6i) ;

		\draw[->]   (vi) --  (v1i) ;
		\draw[->]   (vi) --  (v2i) ;

		\draw[->]   (v1i) --  (v3i) ;
		\draw[->]   (v1i) -- (v4i) ;
		\draw[->]   (v2i) -- (v5i) ;
		\draw[->]   (v2i) --  (v6i) ;

		\node at (-17.5,-9) {$X_3$} ;
		\node at (-11.5,-9) {$X_2$} ;
		\node at (-5,-9) {$X_1$} ;
		\node at (2,-9) {$X_4$} ;

	\end{tikzpicture}
	\caption{}\label{fig:cstreeB}
    \end{subfigure}
    \hfill
    \begin{subfigure}[b]{0.3\textwidth}
    \centering
    \begin{tikzpicture}[thick,xscale=0.15,yscale=0.25]
      \tikzset{>={Stealth[length=3mm, round]}}
		\node[draw, fill=black!0, inner sep=2pt, rounded corners, minimum width=2pt] (w3) at (6,15)  {\scriptsize 1111};
		\node[draw, fill=black!0, inner sep=2pt, rounded corners, minimum width=2pt] (w4) at (6,13.5) {\scriptsize 1110};
		\node[draw, fill=black!0, inner sep=2pt, rounded corners, minimum width=2pt] (w5) at (6,12) {\scriptsize 1101};
		\node[draw, fill=black!0, inner sep=2pt, rounded corners, minimum width=2pt] (w6) at (6,10.5) {\scriptsize 1100};
		\node[draw, fill=black!0, inner sep=2pt, rounded corners, minimum width=2pt] (v3) at (6,9)  {\scriptsize 1011};
		\node[draw, fill=black!0, inner sep=2pt, rounded corners, minimum width=2pt] (v4) at (6,7.5) {\scriptsize 1010};
		\node[draw, fill=black!0, inner sep=2pt, rounded corners, minimum width=2pt] (v5) at (6,6) {\scriptsize 1001};
		\node[draw, fill=black!0, inner sep=2pt, rounded corners, minimum width=2pt] (v6) at (6,4.5) {\scriptsize 1000};
		\node[draw, fill=black!0, inner sep=2pt, rounded corners, minimum width=2pt] (w3i) at (6,3)  {\scriptsize 0111};
		\node[draw, fill=black!0, inner sep=2pt, rounded corners, minimum width=2pt] (w4i) at (6,1.5) {\scriptsize 0110};
		\node[draw, fill=black!0, inner sep=2pt, rounded corners, minimum width=2pt] (w5i) at (6,0) {\scriptsize 0101};
		\node[draw, fill=black!0, inner sep=2pt, rounded corners, minimum width=2pt] (w6i) at (6,-1.5) {\scriptsize 0100};
		\node[draw, fill=black!0, inner sep=2pt, rounded corners, minimum width=2pt] (v3i) at (6,-3)  {\scriptsize 0011};
		\node[draw, fill=black!0, inner sep=2pt, rounded corners, minimum width=2pt] (v4i) at (6,-4.5) {\scriptsize 0010};
		\node[draw, fill=black!0, inner sep=2pt, rounded corners, minimum width=2pt] (v5i) at (6,-6) {\scriptsize 0001};
		\node[draw, fill=black!0, inner sep=2pt, rounded corners, minimum width=2pt] (v6i) at (6,-7.5) {\scriptsize 0000};

		\node[draw, fill=orange!60, inner sep=2pt, rounded corners, minimum width=2pt] (w1) at (-2,14.25) {\scriptsize 111};
		\node[draw, fill=blue!00, inner sep=2pt, rounded corners, minimum width=2pt] (w2) at (-2,11.25) {\scriptsize 110};
		\node[draw, fill=orange!0, inner sep=2pt, rounded corners, minimum width=2pt] (v1) at (-2,8.25) {\scriptsize 101};
		\node[draw, fill=orange!60, inner sep=2pt, rounded corners, minimum width=2pt] (v2) at (-2,5.25) {\scriptsize 100};
		\node[draw, fill=cyan!60, inner sep=2pt, rounded corners, minimum width=2pt] (w1i) at (-2,2.25) {\scriptsize 011};
		\node[draw, fill=cyan!60, inner sep=2pt, rounded corners, minimum width=2pt] (w2i) at (-2,-0.75) {\scriptsize 010};
		\node[draw, fill=cyan!60, inner sep=2pt, rounded corners, minimum width=2pt] (v1i) at (-2,-3.75) {\scriptsize 001};
		\node[draw, fill=cyan!60, inner sep=2pt, rounded corners, minimum width=2pt] (v2i) at (-2,-6.75) {\scriptsize 000};

		\node[draw, fill=green!60, inner sep=2pt, rounded corners, minimum width=2pt] (w) at (-8,12.75) {\scriptsize 11};
		\node[draw, fill=cyan!0, inner sep=2pt, rounded corners, minimum width=2pt] (v) at (-8,6.75) {\scriptsize 10};
		\node[draw, fill=green!60, inner sep=2pt, rounded corners, minimum width=2pt] (wi) at (-8,0.75) {\scriptsize 01};
		\node[draw, fill=cyan!0, inner sep=2pt, rounded corners, minimum width=2pt] (vi) at (-8,-5.25) {\scriptsize 00};

		\node[draw, fill=yellow!0, inner sep=2pt, rounded corners, minimum width=2pt] (r) at (-14,9.75) {\scriptsize 1};
		\node[draw, fill=yellow!0, inner sep=2pt, rounded corners, minimum width=2pt] (ri) at (-14,-1.75) {\scriptsize 0};

		\node[draw, fill=black!0, inner sep=2pt, rounded corners, minimum width=2pt] (I) at (-20,3) {\scriptsize r};

		\draw[->]   (I) --    (r) ;
		\draw[->]   (I) --   (ri) ;

		\draw[->]   (r) --   (w) ;
		\draw[->]   (r) --   (v) ;

		\draw[->]   (w) --  (w1) ;
		\draw[->]   (w) --  (w2) ;

		\draw[->]   (w1) --   (w3) ;
		\draw[->]   (w1) --   (w4) ;
		\draw[->]   (w2) --  (w5) ;
		\draw[->]   (w2) --  (w6) ;

		\draw[->]   (v) --  (v1) ;
		\draw[->]   (v) --  (v2) ;

		\draw[->]   (v1) --  (v3) ;
		\draw[->]   (v1) --  (v4) ;
		\draw[->]   (v2) --  (v5) ;
		\draw[->]   (v2) --  (v6) ;

		\draw[->]   (ri) --   (wi) ;
		\draw[->]   (ri) -- (vi) ;

		\draw[->]   (wi) --  (w1i) ;
		\draw[->]   (wi) --  (w2i) ;

		\draw[->]   (w1i) --  (w3i) ;
		\draw[->]   (w1i) -- (w4i) ;
		\draw[->]   (w2i) --  (w5i) ;
		\draw[->]   (w2i) --  (w6i) ;

		\draw[->]   (vi) --  (v1i) ;
		\draw[->]   (vi) --  (v2i) ;

		\draw[->]   (v1i) --  (v3i) ;
		\draw[->]   (v1i) -- (v4i) ;
		\draw[->]   (v2i) -- (v5i) ;
		\draw[->]   (v2i) --  (v6i) ;

		\node at (-17.5,-9) {$X_3$} ;
		\node at (-11.5,-9) {$X_2$} ;
		\node at (-5,-9) {$X_1$} ;
		\node at (2,-9) {$X_4$} ;

	\end{tikzpicture}
	\caption{}\label{fig:cstreeC}
    \end{subfigure}

\caption{Three staged trees representing models on four jointly distributed binary variables $X_1,X_2,X_3,X_4$. Two first two trees are model equivalent, while the third one is not.}
\label{fig: equivalent staged trees}
\end{figure}
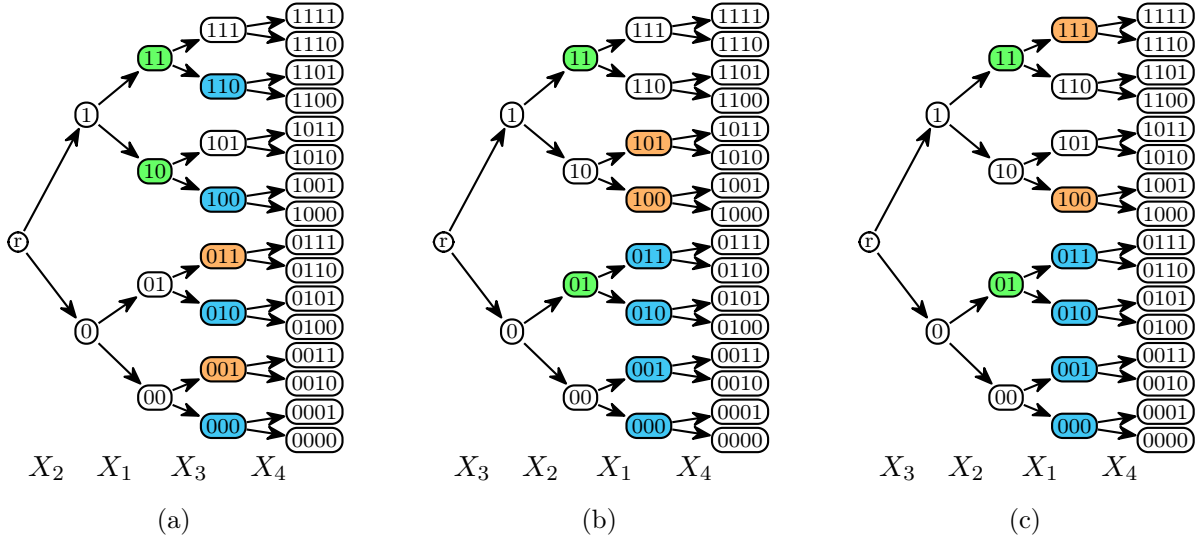

\begin{example}[Equivalence of staged trees] \label{ex: staged tree model equivalence}
Characterizing model equivalence of staged trees is a challenging open problem that many have considered in special cases \cite{gorgen2018discovery, gorgen2018equivalence, duarte2021representation}. The~description of the vanishing ideal up to saturation gives a simple method for determining whether or not two staged trees are model-equivalent. Indeed, consider two staged trees $(T,c)$ and $(T,c')$ supported on the same tree with parametrization $\alpha_T$, the linear spaces $\CC L_c$ and $\CC L_{c'}$ generated by the stages via \eqref{eq:Staged}, and respective ideals of model invariants $\SR J$ and~$\SR J'$. Then $\SR J : S \subseteq \SR J' : S$ if and only if every generator of $\SR J$ evaluates under $\alpha_T$ to an element of $\CC L_{c'}$. This is an easy computation which yields an exact model equivalence test without computing the vanishing ideals of the models. Instead, it uses the vanishing ideals of the parameter spaces which are linear and thus do not require the full complexity of Gröbner bases.

For example, this test decides in the blink of an eye that the staged tree models in \Cref{fig:cstreeA,fig:cstreeB} are model-equivalent, and that neither of these two trees is equivalent to \Cref{fig:cstreeC}. Our computations are based on the code provided with \cite{StagedToric}. The equivalence of the trees in~\Cref{fig:cstreeA,fig:cstreeB} was verified in \cite{duarte2021representation}. Our test naturally also provides certificates that distinguish inequivalent models. For example, the invariant $p_{1001} \, p_{1010} - p_{1000} \, p_{1011}$ of \Cref{fig:cstreeB} does not lie in the vanishing ideal of \Cref{fig:cstreeC}.
\end{example}

\begin{remark}
\label{rem: disc interventions}
In \cite{duarte2020algebraic}, Duarte and Solus extended the notion of interventional models to staged trees. One can show that interventional staged tree models (and hence discrete interventional DAG models) are also ambirational.
\end{remark}

\subsection{Lyapunov models}

Let $\Stab_n \subseteq \BB R^{n \times n}$ be the set of \emph{stable} matrices, i.e., matrices all of whose eigenvalues have negative real part. Then for every $M \in \Stab_n$ and $C \in \PD_n$ there is a unique solution $\Sigma \in \Sym_n$ to the \emph{Lyapunov equation}
\begin{equation}
  \label{eq:Lyap}
  M \Sigma + \Sigma M^\T = -C.
\end{equation}
This solution is positive definite; see \cite{IdentLyapunov} and its references.
Consider a linear space $\CC L \subseteq \Stab_n$ and fix $C \in \PD_n$. Collecting these solutions $\Sigma$ as $M \in \CC L$ varies defines a statistical model $\CC M(\CC L; C)$. The matrices $\Sigma$ are covariance matrices of stationary Ornstein--Uhlenbeck processes subject to the constraints $\CC L$ on the drift matrix~$M$. In case $\CC L$ is defined by vanishing constraints $m_{ij} = 0$ corresponding to non-edges in a directed graph, this yields the \emph{graphical continuous Lyapunov model} \cite{Lyapunov}. As stochastic processes in equilibrium, Lyapunov models incorporate an implicit temporal perspective which alleviates some of the difficulties in modeling feedback loops that arise with Bayesian networks. Notably, cyclic Lyapunov models are globally identifiable \cite{IdentLyapunov}, while SEMs based on a cyclic graph are not \cite[Proposition~16.2.4]{Sullivant}.

The parametrization map of Lyapunov models is constructed as follows. Since \eqref{eq:Lyap} is a linear matrix equation in~$\Sigma$, a standard vectorization trick in matrix analysis rewrites it as
\begin{equation}
  \label{eq:LyapSigma}
  B(M) \, \vec(\Sigma) = -\vec(C),
\end{equation}
where $B(M) = I_n \otimes M + M \otimes I_n$ is a sum of Kronecker products. If $M$ is stable, then so is $B(M)$ and in particular it is invertible. Hence, the entries of $\Sigma$ can be expressed using Cramer's rule as rational functions in $M$ and~$C$.
Parameter recovery can be attempted by reading \eqref{eq:Lyap} as a linear equation in~$M$. This leads to a similar vectorization
\begin{equation}
  \label{eq:LyapM}
  (\Sigma \otimes I_n + (I_n \otimes \Sigma) K_n) \, \vec(M) = -\vec(C),
\end{equation}
where $K_n$ is the \emph{commutation matrix} satisfying $\vec(M^\T) = K_n \, \vec(M)$. However, since $\Sigma$ and~$C$ are symmetric, this linear system has redundant rows, so this matrix will not be invertible. As explained in \cite{IdentLyapunov}, these redundancies can be removed. Their main result \cite[Theorem~32]{IdentLyapunov} shows that if $G$ is a simple directed graph (not necessarily acyclic), then its graphical continuous Lyapunov model is globally identifiable. The general proof, as written, does not directly translate to our framework. Nevertheless, it can be observed to apply in individual~cases.

\begin{example}[Lyapunov DAG model] \label{ex:Lyapunov}
Consider the complete DAG with all self-loops which imposes the following support constraints on the drift matrix $M$:

\begin{sidebyside}
\centering
\begin{sidebox}{0.3\linewidth}
\centering\noindent%
\begin{tikzpicture}[thick, scale=0.6]
  \tikzset{>={Stealth[length=3mm, round]}}
  \node[circle, draw, inner sep=1pt, minimum width=1pt] (1) at (0,0)  {$1$};
  \node[circle, draw, inner sep=1pt, minimum width=1pt] (2) at (0,-2) {$2$};
  \node[circle, draw, inner sep=1pt, minimum width=1pt] (3) at (2,-1) {$3$};

  \draw (1) edge [->, very thick, loop left] (1);
  \draw (2) edge [->, very thick, loop left] (2);
  \draw (3) edge [->, very thick, loop right] (3);

  \draw[->, very thick] (1) -- (2);
  \draw[->, very thick] (1) -- (3);
  \draw[->, very thick] (2) -- (3);
\end{tikzpicture}
\end{sidebox}%
\begin{sidebox}{0.6\linewidth}
\vspace{-0.5\baselineskip}
\[
  \CC L = \Set{\begin{pmatrix}
  m_{11} & 0 & 0 \\
  m_{21} & m_{22} & 0 \\
  m_{31} & m_{32} & m_{33}
  \end{pmatrix} : m_{ji} \in \BB R}
\]
\end{sidebox}
\end{sidebyside}

\noindent%
Note that we adhere to the convention of \cite{IdentLyapunov} that $m_{ji}$ represents the edge $i \rightarrow j$, but unlike \cite[Example~22]{IdentLyapunov}, which also discusses this example, we orient the edges from low to high numbers.
For simplicity, we will fix $C = I_3$. The parametrization map is then given by Cramer's rule applied to the linear system~\eqref{eq:LyapSigma}:

\begingroup
\vspace{-\baselineskip}
\scriptsize
\arraycolsep=2pt
\begin{gather*}
  \left(\begin{array}{ccc|ccc|ccc}
  2 m_{11} & 0 & 0 & 0 & 0 & 0 & 0 & 0 & 0 \\
  m_{12} & m_{11}+m_{22} & 0 & 0 & 0 & 0 & 0 & 0 & 0 \\
  m_{13} & m_{23} & m_{11}+m_{33} & 0 & 0 & 0 & 0 & 0 & 0 \\ \hline
  m_{12} & 0 & 0 & m_{11}+m_{22} & 0 & 0 & 0 & 0 & 0 \\
  0 & m_{12} & 0 & m_{12} & 2 m_{22} & 0 & 0 & 0 & 0 \\
  0 & 0 & m_{12} & m_{13} & m_{23} & m_{22}+m_{33} & 0 & 0 & 0 \\ \hline
  m_{13} & 0 & 0 & m_{23} & 0 & 0 & m_{11}+m_{33} & 0 & 0 \\
  0 & m_{13} & 0 & 0 & m_{23} & 0 & m_{12} & m_{22}+m_{33} & 0 \\
  0 & 0 & m_{13} & 0 & 0 & m_{23} & m_{13} & m_{23} & 2 m_{33}
  \end{array}\right)
  \left(\begin{array}{c}
  \sigma_{11} \\ \sigma_{12} \\ \sigma_{13} \\ \hline \sigma_{12} \\ \sigma_{22} \\ \sigma_{23} \\ \hline \sigma_{13} \\ \sigma_{23} \\ \sigma_{33}
  \end{array}\right) =
  \left(\begin{array}{c}
  -1 \\ 0 \\ 0 \\ \hline 0 \\ -1 \\ 0 \\ \hline 0 \\ 0 \\ -1
  \end{array}\right).
\end{gather*}
\endgroup

\noindent%
Let $B = B(M)$ denote the \emph{negative} of the matrix on the left-hand side. The negation is necessary to give $B$ a positive determinant whenever $M$ is stable, as we show below. Its rows and columns are indexed by pairs $ij \in [3] \times [3]$. Using $B^{kl,ij}$ for the submatrix of $B$ with row $kl$ and column $ij$ deleted, the numerators of the rational functions defining $\sigma_{ij}$ are easy to write out as linear combinations of maximal minors of~$B$:
\begin{equation*}
  \sigma_{ij} = (-1)^{i+j+1} \frac{|B^{11,ij}| + |B^{22,ij}| + |B^{33,ij}|}{|B|}.
\end{equation*}
After cancelling common terms in the numerators and denominators, the $\sigma_{ij}$ are rational functions in the entries of~$M$ whose numerator degrees range between~$0$ and~$5$, having between~$1$ and~$21$~terms.
Following \cite[Example~22]{IdentLyapunov}, rational functions to identify the entries of $M$ can be obtained by considering this full-rank subsystem of~\eqref{eq:LyapM}:

\begingroup
\vspace{-\baselineskip}
\scriptsize
\arraycolsep=2pt
\begin{gather*}
  \left(\begin{array}{ccc|cc|c}
  2 \sigma_{11} & 0 & 0 & 0 & 0 & 0 \\
  \sigma_{12} & \sigma_{11} & 0 & \sigma_{12} & 0 & 0 \\
  \sigma_{13} & 0 & \sigma_{11} & 0 & \sigma_{12} & \sigma_{13} \\ \hline
  0 & 2 \sigma_{12} & 0 & 2 \sigma_{22} & 0 & 0 \\
  0 & \sigma_{13} & \sigma_{12} & \sigma_{23} & \sigma_{22} & \sigma_{23} \\ \hline
  0 & 0 & 2 \sigma_{13} & 0 & 2 \sigma_{23} & 2 \sigma_{33}
  \end{array}\right)
  \left(\begin{array}{c}
  m_{11} \\ m_{21} \\ m_{31} \\ \hline m_{22} \\ m_{32} \\ \hline m_{33}
  \end{array}\right) =
  \left(\begin{array}{c}
  -1 \\ 0 \\ 0 \\ \hline -1 \\ 0 \\ \hline -1
  \end{array}\right).
\end{gather*}
\endgroup

\noindent
Let the matrix be $A = A(\Sigma)$ and label its rows by $kl$, for $k \le l$, and its columns by $ji$, for ${i \rightarrow j} \in E$, in the lexicographic ordering. This matrix is invertible and Cramer's rule with Laplace expansion gives formulas for the $m_{ji}$ similar to those for the~$\sigma_{ij}$.

This establishes a birational isomorphism, but it is far from full. Unlike \Cref{ex:Undirected}, where Cramer's rule is used on a generic matrix and its inverse, $|A|$ and $|B|$ are not mapped to each other under the isomorphism and they are reducible polynomials because of the special structure of the Kronecker products:
\begin{align*}
  |A| &= 8 \, |\Sigma_{1}| |\Sigma_{12}| |\Sigma_{123}|, \\
  |B| &= -8 \, m_{11} m_{22} m_{33} (m_{11} + m_{22})^2 (m_{11} + m_{33})^2 (m_{22} + m_{33})^2.
\end{align*}
We apply \Cref{lemma:ExtendFinite} to make the isomorphism full. To this end, consider the irreducible factors of $|A|$ and $|B|$ given above. Since $M \in \CC L$ is triangular, it is stable if and only if its diagonal entries are negative. Thus $\<S = \monoid{-m_{11}, -m_{22}, -m_{33}, -m_{11}-m_{22}, -m_{11}-m_{33}, -m_{22}-m_{33}}$ is a monoid containing $|B|$ (up to the positive unit~$8$) whose generators are positive on $\CC L \cap \Stab_3$. Similarly $S = \monoid{|\Sigma_1|, |\Sigma_{12}|, |\Sigma_{123}|}$ is generated by positive irreducible polynomials on $\PD_3$ and contains $|A|$ (up to a positive unit). Application of \Cref{lemma:ExtendFinite} introduces new polynomials to the localizing monoids $\<T$ and $T$ which will be assumed to be positive in applications of \Cref{thm:Biri}. To ensure that these new generators do not further limit the parameter space, we show that they are always positive on $\CC L \cap \Stab_3$. In fact, the numerators are positive sums of squares and the denominators are in $\<S$:

\begingroup
\vspace{-\baselineskip}
\scriptsize
\begin{gather*}
  |\Sigma_{1}|   \mapsto \frac{1}{-2\, m_{11}}, \qquad
  |\Sigma_{12}|  \mapsto \frac{m_{12}^2 + (m_{11} + m_{22})^2}{4\, m_{11} m_{22} (m_{11}+m_{22})^2}, \\
  |\Sigma_{123}| \mapsto \frac1{-8\, m_{11} m_{22} m_{33} (m_{11}+m_{22})^2 (m_{11}+m_{33})^2 (m_{22}+m_{33})^2} \cdot {} \\
  \Big((m_{11}^2 m_{22} + m_{11} m_{22}^2 + m_{11}^2 m_{33} + 2 m_{11} m_{22} m_{33} + m_{22}^2 m_{33} + m_{11} m_{33}^2 + m_{22} m_{33}^2)^2 + {} \\
  (m_{22} m_{23} m_{33} - m_{12} m_{13} m_{22} + m_{11}^2 m_{23} + m_{12}^2 m_{23} + m_{11} m_{22} m_{23} + m_{12} m_{13} m_{33} + m_{11} m_{23} m_{33})^2 + {} \\
  (m_{11} m_{13} m_{22} + m_{13} m_{22}^2 - m_{12} m_{22} m_{23} + m_{11} m_{13} m_{33} + m_{13} m_{22} m_{33} - m_{12} m_{23} m_{33})^2 + {} \\
  (m_{11} m_{12} m_{22} + m_{11} m_{12} m_{33} + m_{12} m_{22} m_{33} + m_{12} m_{33}^2)^2\Big).
\end{gather*}
\endgroup

\noindent%
These sum of squares representations were found by the \TT{PolynomialSumOfSquaresList} function of \TT{Mathematica} \citesoft{Mathematica}. This is also true in the other direction: the images of $m_{11}$, $m_{22}$ and $m_{33}$ under the parameter recovery map are negative sums of squares with denominators from~$S$. This completes the construction of an affine birational isomorphism from the complete DAG which bijectively maps $\CC L \cap \Stab_3$ to the Lyapunov model~$\CC M(\CC L; I_3) = \PD_3$.

\begin{figure}
\begin{subfigure}[t]{0.35\linewidth}
\centering
\includegraphics[width=0.95\linewidth]{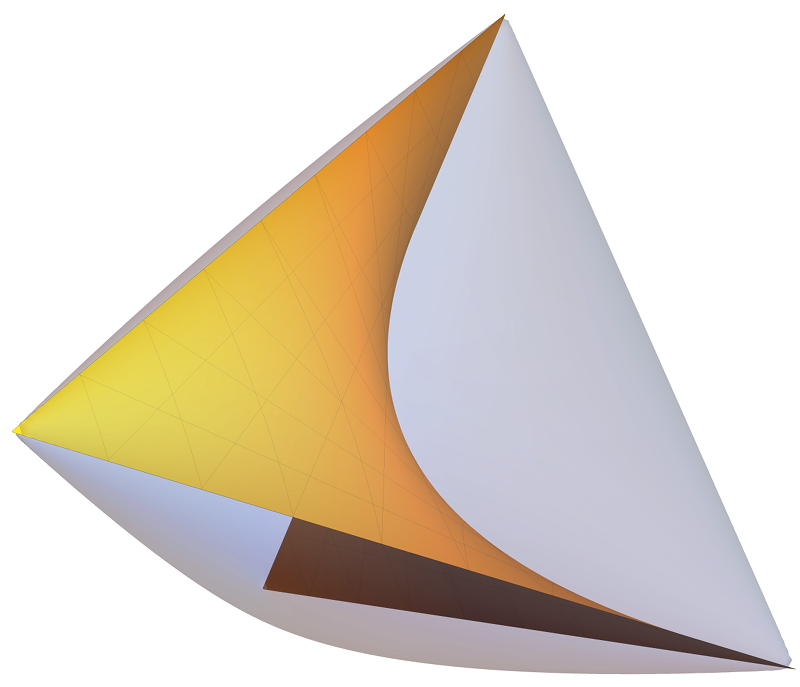}
\caption{Setting $m_{31} = 0$.}
\label{fig:Lyapunov:a}
\end{subfigure}
\hfill
\begin{subfigure}[t]{0.3\linewidth}
\centering
\includegraphics[width=0.9\linewidth]{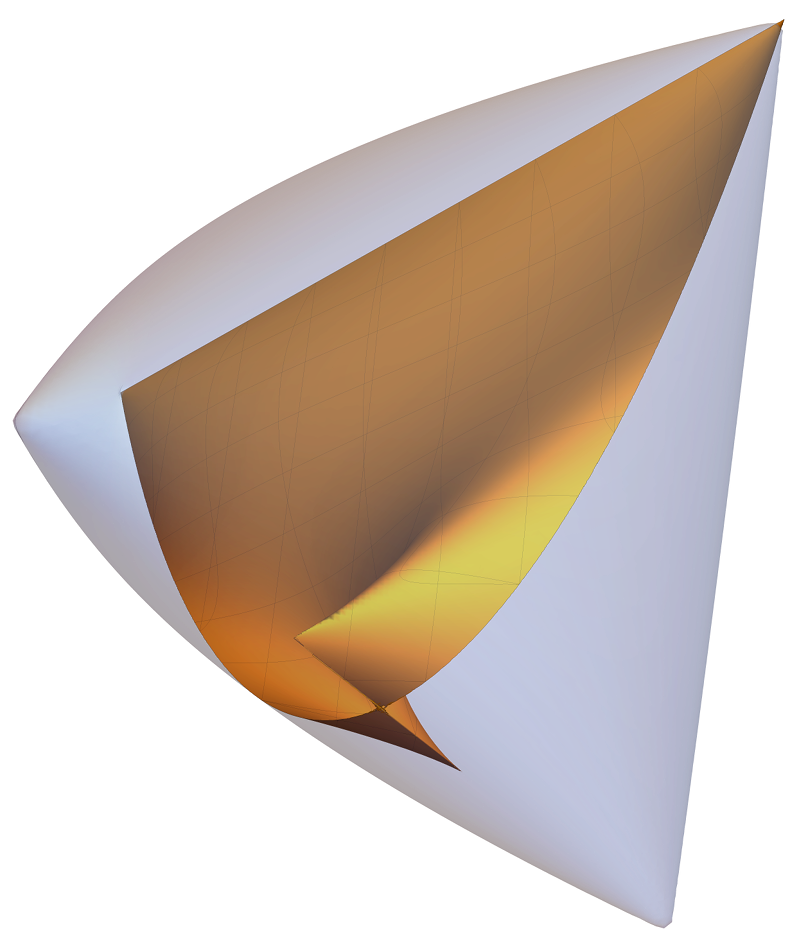}
\caption{Setting $m_{31} = m_{32}$.}
\label{fig:Lyapunov:c}
\end{subfigure}
\hfill\hspace{1em}\begin{subfigure}[t]{0.3\linewidth}
\centering
\includegraphics[width=0.95\linewidth]{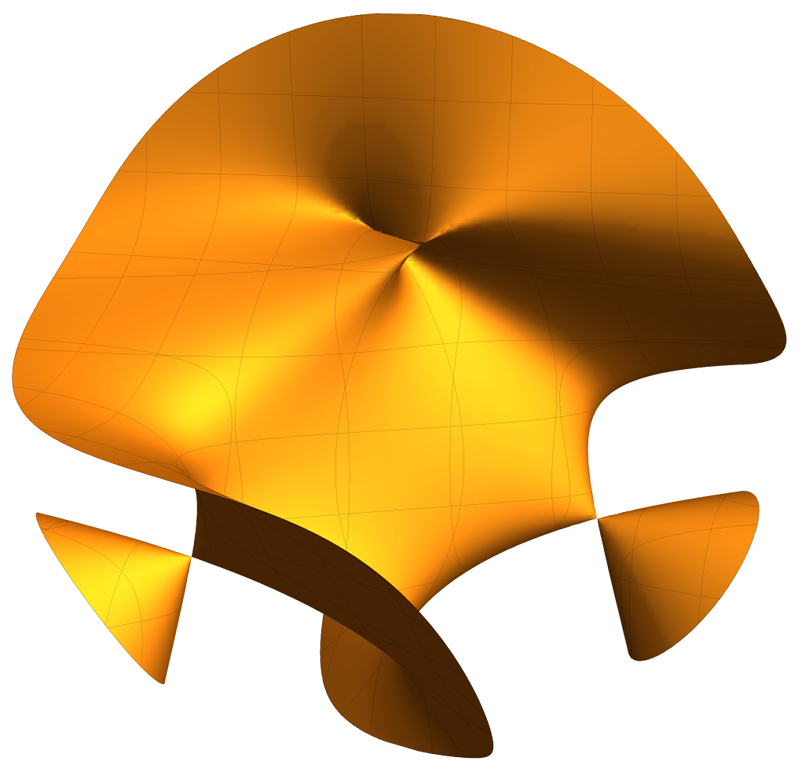}
\caption{Hypersurface of $m_{31} = m_{32}$.}
\label{fig:Lyapunov:d}
\end{subfigure}
\caption{Submodels of the complete DAG Lyapunov model on $3$ vertices.}
\label{fig:Lyapunov}
\end{figure}

\Cref{fig:Lyapunov} shows certain submodels which can be studied using this isomorphism. Note that the submodel in \Cref{fig:Lyapunov:a} with $m_{31}=0$ is really cut out of $\PD_3$ by the irreducible quintic form
\begin{gather*}
  \sigma_{11} \sigma_{12}^2 \sigma_{13} \sigma_{22} - \sigma_{11}^2 \sigma_{13} \sigma_{22}^2 - \sigma_{11} \sigma_{12}^3 \sigma_{23} + \sigma_{11} \sigma_{12} \sigma_{13}^2 \sigma_{23} + \sigma_{11}^2 \sigma_{12} \sigma_{22} \sigma_{23} + {} \\
  \sigma_{12} \sigma_{13}^2 \sigma_{22} \sigma_{23} - \sigma_{11}^2 \sigma_{13} \sigma_{23}^2 - 2 \sigma_{12}^2 \sigma_{13} \sigma_{23}^2 + \sigma_{11} \sigma_{13} \sigma_{22} \sigma_{23}^2 - \sigma_{11} \sigma_{12}^2 \sigma_{13} \sigma_{33} - {} \\
  \sigma_{11} \sigma_{13} \sigma_{22}^2 \sigma_{33} + \sigma_{11}^2 \sigma_{12} \sigma_{23} \sigma_{33} + \sigma_{12}^3 \sigma_{23} \sigma_{33} = 0.
\end{gather*}
In the elliptope (when $\sigma_{11} = \sigma_{22} = \sigma_{33} = 1$), the quintic factors as $(\sigma_{13} - \sigma_{12} \sigma_{23}) \cdot (1 - \sigma_{12} \sigma_{13} \sigma_{23})$ which explains why the figure coincides with the conditional independence model $\CI{1,3|2}$ and thus with the Bayesian network model of the complete DAG with $\lambda_{13} = 0$. But~this is not true in the 6-dimensional cone~$\PD_3$ as $\CI{1,3|2}$ does not hold on the entire submodel defined by~$m_{31}=0$. By~\cite{LyapunovCI} the conditional independences satisfied by a Lyapunov DAG model are determined by its marginal independences and do not in general characterize the model geometrically (in~contrast to \Cref{ex:Undirected} and \Cref{ex:Bayesian}). Nevertheless, \Cref{thm:Biri} provides a useful Markov property for this model in the form of checkable constraints. It is still a mystery how these constraints may be read off more directly from the graph in a manner similar to the Markov properties of Bayesian networks. The high degree of non-linearity of Lyapunov models becomes more apparent in the ``colored'' model with $m_{31}=m_{32}$ depicted in~\Cref{fig:Lyapunov:c,fig:Lyapunov:d}.
\end{example}

\begin{remark}
It would be interesting to study the Lyapunov 3-cycle in detail. In this case (and, in fact, also in \Cref{ex:Lyapunov}), $|B|$ factors into principal minors of the \emph{Hurwitz matrix} $H$ of $M$, specifically $|B| = 8\,|H_{12}||H_{123}|$. Recall that $M$ is stable if and only if the leading principal minors of $H$ are all positive. Further analysis is greatly complicated by the appearance of large irreducible factors in the images of the leading principal minors of both $H$ and $\Sigma$ under the birational maps. It is not clear whether these polynomials are positive on $\PD_3$ and $\CC L \cap \Stab_3$, respectively, or if they impose additional constraints when applying \Cref{thm:Biri}.
The irreducible factors of $|A|$ and $|B|$ and their positivity are a subject of ongoing research; see also~\cite[Example~6]{IdentLyapunov}.
\end{remark}

\section{Applications and discussion}
\label{sec: stats}

The examples given in the previous section illustrate potential applications of birational implicitization. We distill them here into greater generality.

\subsection{Markov properties}
\label{subsec: markov properties}

In graphical modeling, the term \emph{Markov property} refers to an implicit description of a statistical model, usually via conditional independence relations which can be read off from the graph. More broadly, a Markov property for a statistical model is simply an (ideally finite) collection of model-defining constraints. \Cref{thm:Biri} provides a Markov property for any ambirational statistical model.
\begin{definition} \label{def:AlgMarkovProp}
The \emph{ambirational Markov property} of an ambirational statistical model $\CC M = \alpha( \<{\CC M})$ is the set of polynomial equations, inequalities and inequations defining $\CC M$ in \Cref{thm:Biri}.
\end{definition}
These model-defining equations and inequalities may be tested using (possibly incomplete) U-statistics as explained in \cite{sturma2024testing}. In addition to model membership tests, this provides natural tools for constraint-based structure learning methods in the context of graphical modeling.

\Cref{ex:Undirected,ex:Bayesian} show that the ambirational Markov properties recover the well-known pairwise Markov property and pairwise local Markov property for Gaussian undirected graphical models and Bayesian networks, respectively. Similarly, the ambirational Markov property for staged trees described in \Cref{subsec:StagedTree} recovers the well-known ordered Markov property when restricted to the case of (discrete) Bayesian networks.
In \Cref{ex:ColoredDAG,ex:ErrorIntervention}, birational implicitization yields semialgebraic model descriptions which do not necessarily revolve around conditional independence but nevertheless have a known statistical interpretation. %

Most notably, at one of the newest frontiers of graphical modeling research, the ambirational Markov property yields model-defining constraints for Lyapunov models (\Cref{ex:Lyapunov}) for which neither the statistical meaning nor a combinatorial interpretation is available. We hope that the ability to write down these polynomial relations stimulates research on how to explain~them.

\subsection{Model equivalence}
\label{subsec: model equivalence}

In statistical and machine learning, it is commonplace to encode model constraints (such as dependencies) via simple combinatorial structures such as graphs. \emph{Representation learning} is the problem of recovering from data a good (if not the best) combinatorial structure representing the model constraints apparent in the data.
In causality and graphical models, representation learning specializes to the problem of \emph{structure learning}: to choose from a finite collection of graphs the one whose model fits a given sample best. Since the goodness of fit depends on the model, the best graph can only be recovered up to model equivalence. Thus it becomes an important subtask to \emph{distinguish} graphical models from each other and to classify graphs according to model equivalence. These problems are well-studied for classical types of models, such as Bayesian networks, and our methods recover the standard Markov properties which provide a complete solution to the model equivalence problem. But \Cref{thm:Biri,thm:VanishingIdeal} transparently accommodate more general statistical assumptions on the parameters, such as non-linear equations, inequalities or simply different types of linear equations (as in \Cref{ex:RCON,ex:ColoredDAG,ex:MonotoneIntervention}).

Suppose that $\CC M_1 = \alpha_1(\<{\CC M}_1)$ and $\CC M_2 = \alpha_2(\<{\CC M}_2)$ are two ambirational models in the same ambient space~$\BB R^n$. From the semialgebraic descriptions of their parameter spaces, \Cref{thm:Biri} derives finite Markov properties. In concrete cases, it may be possible to certify that the constraints of one model are satisfied on the other and vice versa to establish model equivalence. They can also be used to certify inequivalence by exhibiting a point in one model which violates the other's Markov property.
Since the models are ambirational, the equivalence tests can be performed in the common model space or in the respective parameter spaces. Indeed, a constraint $F_1 \mathrel{\bowtie} 0$ from \Cref{thm:Biri} with ${\bowtie} \in \Set{ {=}, {\ge}, {\not=} }$, which is known to be valid for $\CC M_1$, is valid for $\CC M_2$ if and only if $\alpha_2^*(F_1) \mathrel{\bowtie} 0$ is valid on $\<{\CC M}_2$. In algebraic statistics, parameter spaces are generally much simpler than their models and the birational isomorphism makes it possible to exploit this asymmetry in complexity. We have the following test for model equivalence of ambirational statistical models as an immediate corollary to~\Cref{thm:Biri}.

\begin{corollary} \label{cor:ModelEquivalence}
Let $\CC M_1 = \alpha_1(\<{\CC M}_1)$ and $\CC M_2 = \alpha_2(\<{\CC M}_2)$ be two ambirational models in the same ambient space~$\BB R^n$.
Then $\CC M_1 = \CC M_2$ if and only if  $\alpha_j^*(F_i) \mathrel{\bowtie} 0$ is valid on $\<{\CC M}_j$ for all constraints $F_i \mathrel{\bowtie} 0$ defining $\CC M_i$ as in \Cref{thm:Biri}, for $i\neq j$ with $i, j\in \{1,2\}$.
\end{corollary}

While \Cref{cor:ModelEquivalence} is a computational test for two specific models, the polynomial equations obtained via \Cref{thm:Biri} may be used to give proofs of model inequivalence (and hence, structural identifiability results). Specifically, one can give such proofs by arguing that a certain constraint on the first model, obtained via \Cref{thm:Biri}, cannot be in the vanishing ideal of the second model. This is the essence of the structural identifiability proofs introduced in \cite{ColoredDAG} for families of colored Gaussian DAG models. For general colored DAGs, such proof mechanisms could possibly be applied to yield a graphical characterization of model equivalence.

\subsection{Vanishing ideals}
\label{sec:VanishingIdeal}

From an algorithmic point of view, testing whether $F \mathrel{\bowtie} 0$ on a semialgebraic set~$\CC M$ is quite difficult in general. It is the model problem for the complexity class $\forall\BB R$ which embodies the \emph{universal theory of the reals}; cf.~\cite{BeyondETR,CompendiumETR}. Numerical methods from polynomial optimization are predominantly used in practice~\cite{SOS}. Fortunately, statistical models often have an additional feature: the inequalities are not decisive for model equivalence --- the vanishing ideal is. \Cref{thm:VanishingIdeal} provides a general tool for finding this vanishing ideal via saturation as opposed to elimination.

The classical approach in algebraic statistics is to decompose a statistical model $\CC M = \CC V \cap \CC K$ into a (complex) algebraic variety $\CC V$ and a full-dimensional semialgebraic set~$\CC K$. The inequalities defining $\CC K$ represent general statistical assumptions about the \emph{type} of distribution and are the same for all models of the same type. For Gaussian distributions, $\CC K$ is the cone of positive definite (covariance or concentration) matrices; for discrete distributions, it is the non-negative orthant (and $\CC V$ always contains the constraint that all atomic probabilities sum to~$1$). In these important cases $\CC K$ is even convex.
The main point is that for important types of statistical models $\CC M$ is Zariski-dense in $\CC V$. Thus, whether two such models $\CC M_1 = \CC V_1 \cap \CC K$ and $\CC M_2 = \CC V_2 \cap \CC K$ are equal can be decided by comparing their vanishing ideals, ignoring the inequalities encoded in~$\CC K$. A~convenient sufficient condition for Zariski density of $\CC M$ in $\CC V$ is that $\CC V$ is irreducible and of the same dimension as~$\CC M$, so that there are no extraneous components outside of $\CC K$. All models discussed in \Cref{sec:Ambirational}, which do not explicitly introduce extra inequalities, meet this condition, and this is easy to check because their parameter spaces are defined via linear equations and inequalities with rational coefficients. This explains the success of the decomposition~approach.

\begin{figure}[t]
\begin{minted}{macaulay2}
needsPackage "GraphicalModels";

V = {0,1,2,3,4,5};
G = digraph(V,{{0,2},{0,3},{1,2},{1,3},{2,3},{3,4},{0,5},{1,5},{2,5},{3,5},{4,5}});
R = gaussianRing G;
S = covarianceMatrix R;
allE = set(flatten for i in 0..#V-1 list for j in i+1..#V-1 list (V#i,V#j));

-- Vanishing ideal via built-in elimination method: not finished after 20 minutes
time I1 = gaussianVanishingIdeal R;

-- Vanishing ideal via saturation: 0.186855 seconds
time (
  prs = for i in V list (
    P := toList parents(G, i);
    if #P == 0 then 1 else det submatrix(S, P, P)
  );
  J = ideal for ij in toList(allE-set(edges G)) list (
    P := toList parents(G, ij#1);
    det submatrix(S, {ij#0}|P, {ij#1}|P)
  );
  I2 = fold(saturate, J, prs);
);
\end{minted}
\caption{Elimination vs.\ saturation for computing the vanishing ideal of a Gaussian DAG model.}
\label{fig:Code}
\end{figure}

The choice of the algebraic variety $\CC V$ may vary depending on the implicit model representation of interest. For instance, in graphical models, $\CC V$ may be the zero locus of the conditional independence ideal associated to a Markov property which differs from the ambirational one, as considered in \cite{geiger2006toric, SullivantGaussianBN}. In these situations, the conditional independence ideal of interest, say $\SR C$, typically contains the ideal $\SR J$ of model-defining equations recovered via \Cref{thm:VanishingIdeal}. Hence~$\SR J \subseteq \SR C \subseteq \SR I \subseteq \SR J:S$, where $\SR I$ is the vanishing ideal. Saturation then establishes $\SR J:S = {\SR C:S} = \SR I$ which proves that $\SR C$ saturates to the vanishing ideal as well.

Among other results of this nature, this yields a swift proof of Sullivant's conjecture \cite[Conjecture~3.3]{SullivantGaussianBN} that saturating the global conditional independence ideal of a Gaussian DAG model at the principal minors results in the vanishing ideal. This insight substantially improves the performance of vanishing ideal computations in practice. The \Macaulay2 \citesoft{M2} code in \Cref{fig:Code} uses the built-in \TT{GraphicalModels} \citesoft{GraphicalModels} package to compute the vanishing ideal of a 6-vertex graph using elimination but does not terminate after 20 minutes. The alternative procedure based on \Cref{thm:VanishingIdeal} finishes in 0.2 seconds.

The assumption that two ambirational models are distinguishable by their Zariski closures also improves the model equivalence test from \Cref{cor:ModelEquivalence}: only equations from the respective ambirational Markov property need to be transferred via the pullback and checked on the other model's parameter space. As seen throughout \Cref{sec:Ambirational}, the vanishing ideals of parameter spaces are often linear in statistics. Hence, as a generalization of the procedure used in \Cref{ex: staged tree model equivalence}, we obtain an exact model equivalence test which uses the vanishing ideal without computing it. If the parameter spaces are sufficiently simple (e.g., linear), we can circumvent the use (or at least the worst-case complexity) of Gröbner bases.

In some cases, notably Lyapunov models (\Cref{ex:Lyapunov}), the Markov property and parametrization map consist of such long and high-degree polynomials that even plugging a polynomial $F_1$ into the pullback $\alpha_2^*$ can take considerable amounts of time. In these cases, one could try plugging random values into $\alpha_2$ to generate points on the model~$\CC M_2$ which are then evaluated under~$F_1$. The Schwartz--Zippel lemma can be used to bound the probability that $\alpha_2^*(F_1)$ is non-zero on the parameter space~$\<{\CC M}_2$~--- a technique we learned from \cite{AlgebraicPhylogenetics}.

\vspace{1em}

\section*{Acknowledgements}

\setlength{\intextsep}{5pt}%
\setlength{\columnsep}{15pt}%
\begin{wrapfigure}{R}{0.2\linewidth}
\vspace{-.2\baselineskip}%
\centering%
\href{https://doi.org/10.3030/101110545}{%
\includegraphics[width=0.9\linewidth]{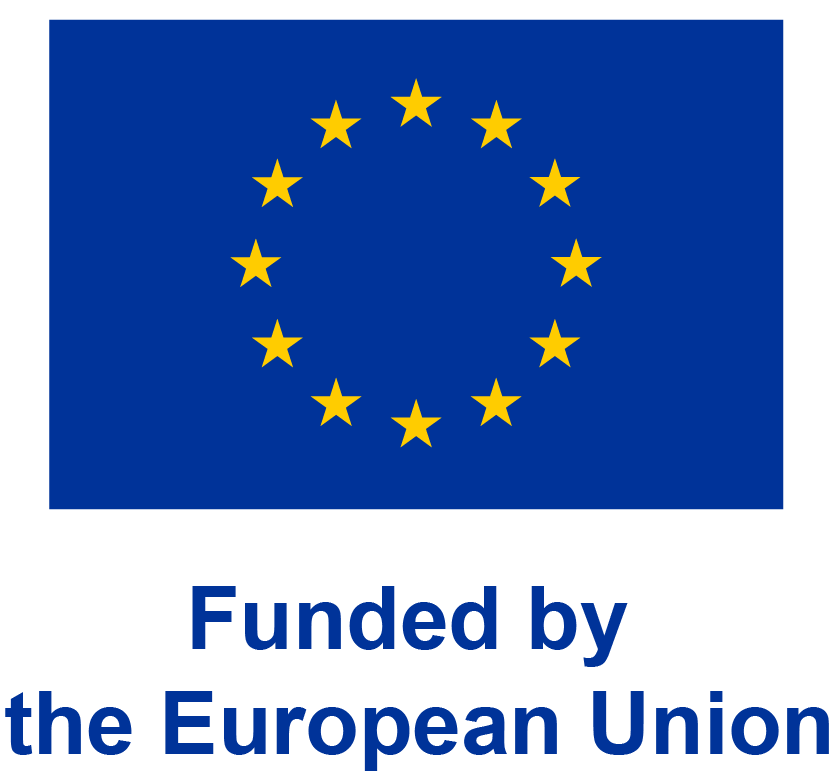}%
}
\end{wrapfigure}
We thank Benjamin Hollering for stimulating discussions. T.B.~and L.S~were partially supported by the Wallenberg Autonomous Systems and Software Program (WASP) funded by the Knut and Alice Wallenberg Foundation. T.B.~has received funding from the European Union's Horizon 2020
research and innovation programme under the Marie Skłodowska-Curie grant
agreement No.~101110545.
L.S.~was further supported by a Starting Grant from the Swedish Research Council (Vetenskapsr\aa{}det), the G\"oran Gustafsson Foundation's Prize for Young Researchers, and a research pairs grant from the center for Digital Futures at KTH.
All plots were generated with Wolfram~\TT{Mathematica}~\citesoft{Mathematica}.

\bibliographystyle{tboege}
\bibliography{biri}

\let\etalchar\undefined
\nocitesoft{Mathematica, M2, GraphicalModels}
\bibliographystylesoft{tboege}
\bibliographysoft{biri}

\end{document}